\documentclass[12pt]{amsart}
\usepackage{setspace}
\usepackage{verbatim}

\usepackage{amsfonts,fancyhdr,ifthen,bm}
\usepackage{amscd,amssymb} 
\usepackage[letterpaper,left=1.35in,right=1.35in,top=1.35in,bottom=1.35in]{geometry}
 
\usepackage{times}
\usepackage{graphicx}
\usepackage{color}
\usepackage{epsfig}
\usepackage{mathrsfs}
\usepackage{paralist}
\usepackage{comment}
\usepackage{mathtools}
\usepackage{multirow}
\usepackage{tikz-cd}
\usepackage[foot]{amsaddr}

\usepackage{enumitem}

\usepackage{appendix}

\usepackage{tikz,ifthen,calc}
\usetikzlibrary{automata,arrows,positioning,calc}

\newtheorem{thm}{Theorem}[section]

\newtheorem*{thm*}{Theorem}
\newtheorem{prop}[thm]{Proposition}
\newtheorem*{prop*}{Proposition}
\newtheorem{cor}[thm]{Corollary}
\newtheorem*{cor*}{Corollary}

\newtheorem{lem}[thm]{Lemma}
\newtheorem*{lem*}{Lemma}

\newtheorem*{oquest*}{Open Question}

\theoremstyle{remark}
\newtheorem{rmk}[thm]{Remark}

\theoremstyle{remark}
\newtheorem*{rmk*}{Remark}

\theoremstyle{definition}
\newtheorem{defn}[thm]{Definition}

\theoremstyle{definition}
\newtheorem{notat}[thm]{Notation}

\theoremstyle{definition}

\theoremstyle{definition}

\theoremstyle{definition}
\newtheorem*{defn*}{Definition}

\theoremstyle{definition}
\newtheorem{ex}[thm]{Example}

\theoremstyle{definition}
\newtheorem{ctn}[thm]{Construction}

\numberwithin{equation}{section}%numbers equations by section

 \usepackage{chngcntr}
\counterwithin{figure}{section}

\newcommand{\Z}{\mathbb{Z}}
\newcommand{\QQ}{\mathbb{Q}}

\newcommand{\msW}{\mathscr{W}}

\newcommand{\mcH}{\mathcal{H}}
\newcommand{\mcO}{\mathcal{O}}
\DeclareMathOperator{\FFF}{\mathbb{F}}

\newcommand{\Sel}{\textup{Sel}}
\newcommand{\Gal}{\textup{Gal}}

\newcommand{\Hom}{\textup{Hom}}

\newcommand{\isoarrow}{\xrightarrow{\,\,\,\sim\,\,\,}}

\newcommand{\royarrow}{\,\xrightarrow{\,\quad}\,}

\newcommand{\Conj}{\textup{Conj}}

\newcommand{\Ind}{\textup{Ind}}
\newcommand{\cores}{\textup{cor}}
\newcommand{\res}{\textup{res}}
\newcommand{\Cl}{\textup{Cl}\,}

\newcommand{\mcalH}{\mathcal{H}}

\newcommand{\CTP}{\textup{CTP}}

\newcommand{\inv}{\textup{inv}}

\DeclareFontFamily{U}{wncy}{}
\DeclareFontShape{U}{wncy}{m}{n}{<->wncyr10}{}
\DeclareSymbolFont{mcy}{U}{wncy}{m}{n}
\DeclareMathSymbol{\Sha}{\mathord}{mcy}{"58}

\newcommand\restr[2]{{% we make the whole thing an ordinary symbol
  \left.\kern-\nulldelimiterspace % automatically resize the bar with \right
  #1 % the function
  \vphantom{\big|} % pretend it's a little taller at normal size
  \right|_{#2} % this is the delimiter
  }}

\newcommand{\SMod}{\textup{SMod}}

\newcommand{\shap}{\textup{sh}}
\newcommand{\Fstar}{(\Fsep)^{\times}} %previously \overline{F}^{\times}
\newcommand{\Fvstar}{F_v^{s \times}} %previously \overline{F_v}^{\times}
\newcommand{\Fsep}{F^s} %previously F^s
\newcommand{\Fvsep}{F_v^s} %previously F_v^s
\newcommand{\natiso}{\,\xRightarrow{\,\,\,\,\simeq\,\,\,\,}\,}
\newcommand{\natmor}{\,\xRightarrow{\,\,\,\,\,\,\,\,\,\,}\,}
\newcommand{\resprod}{\sideset{}{'}\prod}
\newcommand{\tresprod}[1]{\sideset{}{'} {\textstyle \prod_{#1}}}

\newcommand{\Mod}{\textup{Mod}}

\usepackage{scalerel}
\newcommand{\llop}{\scaleto{\bigg(}{33pt}}
\newcommand{\rlop}{\scaleto{\bigg)}{33pt}}

\author{Adam Morgan}
\email{adam.morgan@glasgow.ac.uk}
\author{Alexander Smith}
\email{asmith13@stanford.edu}
 
\begin{document}
\title{Field change for the Cassels--Tate pairing and applications to class groups}
 
\date{\today}

\maketitle

\begin{abstract}

In \cite{MS21}, we defined a category $\SMod_F$ of finite Galois modules decorated with local conditions for each global field $F$. In this paper, given an extension $K/F$ of global fields, we define a restriction of scalars functor from $\SMod_K$ to $\SMod_F$ and show that it behaves well with respect to the Cassels--Tate pairing. We apply this work to study the class groups of global fields in the context of the Cohen--Lenstra heuristics.
\end{abstract}

\setcounter{tocdepth}{1}

\tableofcontents

\section{Introduction} 
The aim of this paper is twofold. First, we wish to apply the theory of Selmer groups and the Cassels--Tate pairing developed in \cite{MS21} to study class groups of global fields.  
 The idea of using the Cassels--Tate pairing in the study of class groups is due to Flach  \cite[p. 122]{Flach90}, but much of the power of this idea comes from the flexibility of the framework    introduced in  \cite{MS21}.  

That said, some of the results we prove about class groups in this paper rely on theory for the Cassels--Tate pairing that is not in \cite{MS21}. In that paper, we defined a category $\SMod_F$ of finite $G_F$-modules decorated with local conditions for each global field $F$, and we associated a Cassels--Tate pairing of Selmer groups to every short exact sequence in $\SMod_F$. For our work on class groups, it is sometimes necessary to move between the categories $\SMod_K$ and $\SMod_F$ for a given extension $K/F$ of global fields, and we did not give the theory to do this in \cite{MS21}. With this in mind, the second goal of this paper is to explore the relationship between the Cassels--Tate pairings defined for these two categories.   Our main result here  shows that the Cassels--Tate pairing behaves well under a certain restriction of scalars functor from $\SMod_K$ to $\SMod_F$.  

We now go through our work towards these two goals in more detail, using the language of \cite{MS21}. In an effort to be self-contained, we have provided a brief recap of what we need from that paper in Section \ref{ssec:recap}.

\subsection{The behavior of the Cassels--Tate pairing in field extensions} \label{ssec:intro_field}
Take $K/F$ to be an extension of global fields. Given a $G_{K}$-module $M$, we can consider the corresponding induced $G_{F}$-module 
\[\Ind_{K/F}M = \Z[G_F] \otimes_{\Z[G_K]} M.\]
Shapiro's lemma gives a natural isomorphism
\begin{equation*} \label{eq:global_shapiro_map_intro}
\shap \colon H^1(G_{K}, \,M) \isoarrow  H^1(G_{F}, \Ind_{K/F} M),
\end{equation*}
 and we also have a local Shapiro isomorphism
\[\shap_{\text{loc}}\colon \resprod_{w \text{ of } K} H^1(G_w, M) \isoarrow \resprod_{v \text{ of } F} H^1(G_v, \Ind_{K/F} M).\]
We discuss this latter map in Section \ref{ssec:shap_loc}. 
With this set, we obtain an exact functor
\[\Ind_{K/F} \colon \SMod_{K} \royarrow \SMod_{F}\]
via the assignment
\[(M, \msW)\,\mapsto\, \left(\Ind_{K/F} M, \,\shap_{\text{loc}}(\msW)\right).\]
Using a subscript to record which of the categories $\SMod_F$ or $\SMod_K$ we are working in, the global Shapiro isomorphism then gives a natural isomorphism
\begin{equation} \label{sh_intro_functor}
\shap \colon \Sel_K \natiso \Sel_F \circ \Ind_{K/F}.
\end{equation}
Now given a short exact sequence 
\begin{equation} \label{seq_in_SMOD_K}
 E = \big[0 \to M_1 \xrightarrow{\,\,\iota\,\,} M \xrightarrow{\,\,\pi\,\,} M_2 \to 0\big]  
\end{equation}
in $\SMod_K$, we can combine \eqref{sh_intro_functor}  with the explicit identification  between
 $\Ind_{K/F}M_1^\vee$ and $(\Ind_{K/F}M_1)^\vee$ detailed in Definition \ref{def:t_iso},  to view  the Cassels--Tate pairing associated to the induced sequence $\Ind_{K/F}E$ as a pairing 
\[\Sel_K M_2 \times \Sel_K M_1^\vee \longrightarrow \mathbb{Q}/\mathbb{Z}.\]
We then have the following, which constitutes the final basic  result on our version of the Cassels--Tate pairing. For a more precise statement, see Theorem \ref{thm:CT_shap}.

\begin{thm}
\label{thm:fields_main_intro}
 Take any short exact sequence $E$ in $\SMod_K$, and let $\Ind_{K/F}E$ denote the corresponding exact sequence in $\SMod_F$. Then, with the identifications above, we have 
\[\CTP_E=\CTP_{\Ind_{K/F}E}.\]  
\end{thm}

According to Theorem \ref{thm:fields_main_intro}, the functor $\Ind_{K/F}$   preserves the structure on the category $\SMod_K$ provided by the Cassels--Tate pairing. Since $\Ind_{K/F} M$ is typically a richer object than $M$, we can often prove extra identities for pairings over $K$ by changing fields to $F$. Furthermore, as we detail  in Corollary \ref{cor:cor}, a consequence of  Theorem \ref{thm:fields_main_intro} is that corestriction is an adjoint to restriction with respect to the Cassels--Tate pairing, assuming local conditions are handled appropriately. This generalizes a result of Fisher for the pairing associated to an elliptic curve \cite{Fish03}.

\subsection{Results on class groups} \label{ssec:class_intro}
   As mentioned above, after proving Theorem \ref{thm:fields_main_intro} the rest of the paper considers how the theory  may be applied, in particular to the study of class groups. We go through a couple of these applications in detail now.

\subsubsection{Class groups of number fields with many roots of unity} \label{sec_intro_roots_of_unity}

Choose a prime $\ell$, and take $F$ to be a global field of characteristic not equal to $\ell$. Take $\Cl\, F$ to be the class group of $F$, take $H_F$ to be its Hilbert class field, and define
\begin{equation} \label{eq:dual_class_intro}
\Cl^* F = \Hom(\Gal(H_F/F), \, \QQ/\Z).
\end{equation}
Taking 
\begin{equation} \label{eq:reciprocity_map_intro}
\text{rec}_F\colon \Cl\, F \isoarrow \textup{Gal}(H_F/F)
\end{equation}
to be the Artin reciprocity map, we define the reciprocity pairing
\begin{equation}\label{eq:reciprocity_pairing_intro}
\text{RP} \colon \Cl\, F[\ell^{\infty}] \times \Cl^* F[\ell^{\infty}] \to \QQ_{\ell}/\Z_{\ell}
\end{equation}
by the rule $\text{RP}(I, \phi) = \phi(\text{rec}_F(I))$. As Flach observed \cite[p. 122]{Flach90}, this pairing may, in the terminology of \cite[Section 4.2]{MS21}, be packaged as the Cassels--Tate pairing associated to the exact sequence
\begin{equation}
\label{eq:all_rec}
0 \to \Z_{\ell} \to \QQ_{\ell}  \to \QQ_{\ell}/\Z_{\ell} \to 0
\end{equation}
decorated with the unramified local conditions.

From this repackaging, we can correctly predict that the duality identity for the Cassels--Tate pairing might give interesting information about the reciprocity pairing on class groups of fields with many roots of unity; this is the case where a portion of the above sequence is self-dual. The following notation will be useful.
\begin{notat} \label{def:char_chi_alpha}
Let $m$ be a power of $\ell$, and suppose that $F$ contains $\mu_m$. 
Given $\alpha$ in $F^{\times}$,   choose an $m^{\text{th}}$ root $\alpha^{1/m}$ of  $\alpha$. We then denote by  $\chi_{m, \,\alpha}$ the homomorphism
\[\chi_{m, \,\alpha}\colon \Gal\left(F(\alpha^{1/m})/F\right)\to \mu_m\]
defined by
\[\chi_{m, \,\alpha}(\sigma) = \frac{\sigma(\alpha^{1/m})}{\alpha^{1/m}}.\]
\end{notat}

\begin{prop}
\label{prop:cg_sym_intro} 
Choose positive powers $a$ and $b$ of $\ell$, and suppose $F$ contains $\mu_{ab}$. Choose fractional ideals $I$, $J$ of $F$ and $\alpha$, $\beta$ in $F^{\times}$ satisfying
\[(\alpha) = I^{a}\quad\text{and}\quad (\beta) = J^{b} .\]
Then, if $F(\alpha^{1/a})/F$ and $F(\beta^{1/b})/F$ are everywhere unramified, we have
\[\chi_{a, \,\alpha}\left(\textup{rec}_F\left(J\right)\right) = \chi_{b, \,\beta}\left(\textup{rec}_F\left(I\right)\right).\]
\end{prop}

In the case where $\ell$ is odd, this explicit reciprocity law is a consequence of work by Lipnowski, Sawin, and Tsimerman \cite{Lipn20}, who proved it while developing a variant of the Cohen--Lenstra--Martinet heuristics for the class groups of quadratic extensions of number fields containing many roots of unity. We will elaborate further on the connection between our work and \cite{Lipn20} in Section \ref{ssec:LST}.

By combining the results on base change with our theory of  theta groups \cite[Section 5]{MS21}, we can prove a nice supplement to this proposition in the case that $\ell = 2$. Both this proposition and Proposition \ref{prop:cg_sym_intro} will be proved in Section \ref{sec:class_groups}.
\begin{prop}
\label{prop:cg_alt_intro}
Choose a global field $F$ of characteristic other than $2$. Let $m$ be a power of $2$, and let $d$ be a positive proper divisor of $m$. Suppose $F$ contains $\mu_{m^2/d}$, and suppose that $\kappa\colon F\to F$ is a field automorphism that satisfies $\kappa^2 = 1$ and which restricts to the map $\zeta \mapsto \zeta^{-1}$ on $\mu_{m^2/d}$.

Take $I$ to be a fractional ideal of $F$ such that $I^m$ is principal, and choose $\alpha \in F^{\times}$ so  
\[(\alpha) = I^m.\]
Assume that $F(\alpha^{1/m})/F$ is unramified at all places of $F$ and splits completely at all primes dividing $2$. 
We then have
\[\chi_{m,\, \alpha}\left(\textup{rec}_F\left(\kappa I^d\right)\right) = 1.\]
\end{prop}
It is our hope that this result and others like it (see  Proposition \ref{prop:Dm_alt}, for example) will prove to be useful in extending the Cohen--Lenstra--Martinet heuristics to handle $2$-primary parts of class groups.

\subsubsection{Decomposing class groups} \label{sec_decom_class_intro}

In \cite{CoLe84}, Cohen and Lenstra gave a conjectural heuristic describing how the class group $\Cl\, K$ varies as $K$ moves through the quadratic extensions of $\QQ$. These heuristics omitted the $2$-primary part of the class group, and Gerth \cite{Gerth87p} later supplemented that work with a prediction for the distribution of the groups $2(\Cl\, K)  [2^\infty]$. For imaginary quadratic fields, Gerth's prediction is now known to be correct \cite{Smi17}. By considering $2(\Cl\, K)$ rather than $\Cl\, K$, Gerth avoided the subgroup $(\Cl\, K)[2]$, which can be directly computed using genus theory and is typically large.

A related question concerns the behavior of Selmer groups of elliptic curves in quadratic twist families, where one fixes an elliptic curve $E/\mathbb{Q}$ and studies the behavior of the Selmer groups of the quadratic twists $E^\chi/\mathbb{Q}$ as $\chi$ varies over quadratic characters. See e.g. \cite{Heat94, Kane13, Smi17}.

These are analogous problems, but they belong to different genres of arithmetic statistical questions.
 Both kinds of problem start with a fixed $G_{\QQ}$-module $M$.  In the first, one asks how $\Sel_K\, M$ varies as $K$ moves through a family of number fields. In the second, one asks how $\Sel_{\QQ}\, M^{\chi}$ varies as $\chi$ moves through a family of twists.  In the particular examples considered above, $M$ is taken either to be $\QQ/\Z$ with the unramified local conditions, or the divisible Tate module of $E$ with the standard local conditions.

We can often translate between these two kinds of problems. Suppose $K/F$ is a quadratic extension of number fields, and take $\chi$ to be its corresponding quadratic character. If $M$ is a $G_F$-module for which multiplication by $2$ is invertible, the module $\Ind_{K/F} M$ splits as a direct sum
\[\Ind_{K/F} M \,\cong\,  M \oplus M^{\chi}.\]
For class groups, for example, we have a decomposition
\[(\Cl^*K)[\ell^\infty]\,\cong\, \Sel_F\left(\Ind_{K/F} \,\QQ_{\ell}/\Z_{\ell}\right) \, \cong \, ( \Cl^*F)[\ell^\infty]\oplus \textup{Sel}_F(\mathbb{Q}_{\ell}/\mathbb{Z}_{\ell})^{\chi}\]
for every odd prime $\ell$, where the local conditions are set as in Section \ref{ssec:class_decomp}.

If multiplication by $2$ on $M$ is not invertible, the relationship between $\Sel_K \,M$ and $\Sel_F\, M^{\chi}$ can be more complicated. In this case, the Cassels--Tate pairing proves to be a convenient tool for moving from Selmer groups over extensions to Selmer groups of twists. As an example, we will prove the following as the $l=2$ case of Proposition \ref{prop:class_decomp_ex}. 

\begin{prop}
Take $K/F$ to be a quadratic extension of number fields with associated quadratic character $\chi$,  and take $S$ to be the set of places of $F$ not dividing $2$ where $K/F$ is ramified. Suppose there is no element  of $F^{\times}\setminus (F^{\times})^2$ which lies in $(F_v^{\times})^{2}$ at every place $v$ in $S$, and which has even valuation at all remaining finite primes.  

Then we have an exact sequence
\[0 \to \left(\Cl^*K\right)^{\Gal(K/F)} [2^{\infty}]\royarrow (\Cl^*K)[2^{\infty}] \royarrow \Sel (\QQ_2/\Z_2)^{\chi} \to 0\]
of finite $\Gal(K/F)$-modules.
\end{prop}
The first term $ \left(\Cl^*K\right)^{\Gal(K/F)} [2^{\infty}]$ is essentially determined from genus theory, making it an analogue of the subgroup $(\Cl\, K)[2]$ in Gerth's conjecture. Meanwhile, the final term $\Sel(\QQ_2/\Z_2)^{\chi}$, which agrees with $2(\Cl\, K)  [2^\infty]$ in the situation considered by Gerth, is more in line with what the Cohen--Lenstra heuristics can handle. 

The idea of studying the $\ell^{\infty}$ part of the class group of $K$ by studying the module $\Ind_{K/F} \QQ_{\ell}/\Z_{\ell}$ and its irreducible components is not new; see e.g \cite[1.6.4]{MR1749177}, \cite{CoMa90}, or \cite{WaWo21}. 

The novelty of our work lies in the ability to handle ``bad'' choices of $\ell$.

\subsection{Layout of the paper} 
In Section \ref{sec:embeddings} we summarize the notation used throughout the paper and review the material we need from \cite{MS21}.

Section \ref{sec:embedding_invariance} addresses a technical point that we have so far sidestepped, showing that the Cassels--Tate pairing is, in a suitable sense, independent of the choice of embeddings $F^s\hookrightarrow F^s_v$ used to define restriction maps on cohomology. In Section \ref{sec:fields}, we define the functor $\Ind_{K/F} \colon \SMod_K \to \SMod_F$ for a finite separable extension $K/F$ and prove Theorem \ref{thm:fields_main_intro}. These sections rely on some basic properties of group change homomorphisms for which we have not found a suitable reference. We give this necessary background material in Appendix \ref{app:group_change}.

In Section \ref{sec:class_groups}, we prove our symmetry result for class groups of number fields with many roots of unity. As mentioned above, this recovers a result of  Lipnowski, Sawin, and Tsimerman \cite{Lipn20}, who encode the symmetry result using the language of bilinearly enhanced groups. Section \ref{ssec:LST} shows that this bilinearly enhanced group structure may be recovered as instances of the Cassels--Tate pairing and extends this work to $l=2$ using theta groups.  

Section \ref{ssec:class_decomp} contains the material introduced in Section \ref{sec_decom_class_intro}, and Section \ref{sec:Two} gives two miscellaneous applications of the theory of the Cassels--Tate pairing.

\subsection{Acknowledgments} 
We would like to thank Peter Koymans, Wanlin Li, Melanie Matchett Wood, Carlo Pagano, Bjorn Poonen, Karl Rubin, and Richard Taylor for helpful discussion about this work.  

This work was partially conducted during the period the second author served as a Clay Research Fellow. Previously, the author was supported in part by the National Science Foundation under Award No. 2002011.

The first author is supported by the Engineering and Physical
Sciences Research Council  grant EP/V006541/1. During part of the period in which this work was carried out, they were supported by the Max-Planck-Institut für Mathematik in Bonn and the Leibniz fellow programme at the Mathematisches Forschungsinstitut Oberwolfach, and would like to express their thanks for the excellent conditions provided.

\section{Notation and review of \cite{MS21}}
\label{sec:embeddings}
 \subsection{Conventions}
We will keep the same notation as in \cite[Section 2]{MS21}. In an effort to be self-contained, we recall some of our conventions now.

 Throughout, $F$ will denote a global field with fixed separable closure $F^s/F$ and absolute Galois group $G_F=\Gal(\Fsep/F)$.  For each place $v$ of $F$, $F_v$ denotes the completion of $F$ at $v$,  $F_v^s$ denotes a fixed separable closure of $F_v$, and   $G_v$ denotes the Galois group of $F^s_v/F_v$. We write $I_v$ for the inertia subgroup of $G_v$. For each place $v$ we fix an embedding $\Fsep \hookrightarrow \Fvsep$, via which we view $G_v$ as a closed subgroup of $G_F$.
 
Given a topological group $G$, a discrete $G$-module $M$, and $i\geq 0$, $H^i(G, M)$ will denote the continuous group cohomology of $G$ acting on $M$.  We often work with these groups using the language of inhomogeneous $i$-cochains \cite[Section 2]{Neuk08}, the group of which we denote $C^i(G, M)$. We denote by  $Z^i(G,M)$ the subgroup of $C^i(G, M)$ consisting of $i$-cocycles. Given $\phi$ in $C^i(G_F, M)$ or $H^i(G_F, M)$, we write $\phi_v$ as  for the restriction of $\phi$ to $G_v$.

Finally, given a $G_F$-module $M$, we write $M^\vee$ for the dual $G_F$-module $\Hom(M,\, (F^s)^{\times})$, and denote  by $\cup_M$ the cup product constructed from the standard evaluation pairing   $M \times M^{\vee}\rightarrow (\Fsep)^{\times}$  (cf. \cite[Notation 2.2]{MS21}).

\subsection{A review of \cite{MS21}}  \label{ssec:recap}
We briefly recap some of the relevant material from \cite{MS21}. The reader familiar with that work may safely skip this section.

\subsubsection{The category $\SMod_F$}
 In  \cite[Definition 1.1]{MS21} the  category $\SMod_F$ was introduced. Its objects are pairs $(M,\msW)$ consisting of a finite $G_F$-module $M$ of order indivisible by the characteristic of $F$ along with a compact open subgroup $\msW$ of the  restricted product $\tresprod{v} H^1(G_v, M).$ We recall also that:
\begin{itemize}
\item the dual of an object $(M,\msW)$  is the pair $(M^\vee, \msW^\perp)$, where  $\msW^\perp$ is  the orthogonal complement of $\msW$ with respect to the pairing induced by Tate local duality; we denote by $\vee$ the contravariant functor on $\SMod_F$ sending an object to its dual.  
\item a sequence \begin{equation}\label{eq:example_exact}
E = \left[0 \to (M_1,\,\msW_1) \xrightarrow{\,\,\iota\,\,} (M,\, \msW) \xrightarrow{\,\,\pi\,\,} (M_2,\, \msW_2) \to 0\right]
\end{equation} in $\SMod_F$ is said to be exact if  the underlying sequence of $G_F$-modules is exact and if
\[\iota^{-1}(\msW) = \msW_1\quad\text{and}\quad \pi(\msW) = \msW_2.\]
\end{itemize}

  Associated to any object $(M,\msW)$ in $\SMod_F$ is the Selmer group 
  \[\Sel(M,\,\msW) \,=\, \ker\llop H^1(G_F, M) \xrightarrow{\quad} \,\, \resprod_{v\text{ of } F} H^1(G_v, M)  \Big/\msW\rlop.\]
The definition of the morphisms in $\SMod_F$ ensure that the map sending $(M,\msW)$ to its Selmer group defines a functor $\Sel$ from $\SMod_F$ to the category of finite abelian groups. 

\subsubsection{The Cassels--Tate pairing}
The main object of interest in \cite{MS21} is a version of the Cassels--Tate pairing associated to each short exact sequence in $\SMod_F$. 

\begin{thm}[\cite{MS21}, Theorem 1.3]  
Given a short exact sequence $E$ in $\SMod_F$, as depicted in \eqref{eq:example_exact},  there is a bilinear pairing
\begin{equation} \label{cassels_tate_pairing_intro}
\CTP_E\colon \Sel\, M_2\, \times\, \Sel\, M_1^{\vee}\to \QQ/\Z, 
\end{equation}
with left kernel $\pi(\Sel M)$ and right kernel $\iota^\vee(\Sel M^\vee)$.  
\end{thm}

Two basic properties of the pairing that are established in  \cite{MS21} are the following: 
\begin{itemize}
\item  Given $E$ as above, and identifying $M_2$ with $(M_2^{\vee})^{\vee}$ in the standard way, we have  the \textit{duality identity}
\begin{equation} \label{eq:intro_symm}
\CTP_{E}(\phi, \psi) = \CTP_{E^\vee}(\psi, \phi)
\end{equation}
for all $\phi$ in $\Sel \,M_2$ and $\psi$ in $\Sel\, M_1^{\vee}$.  
\item Given a morphism of short exact sequences 
\begin{equation}
\label{eq:nat_setup}
\begin{tikzcd}
E\,=\,\big[\,0 \arrow{r} & M_{1}\arrow{d}{f_1} \arrow{r}{\iota}& M \arrow{d}{ f}    \arrow{r}{\pi}& M_{2} \arrow{d}{ f_2}    \arrow{r} & 0\,\big] \\
E'\,=\,\big[\,0 \arrow{r} & M_{1}' \arrow{r}{\iota'} & M' \arrow{r}{\pi'} & M_{2}' \arrow{r} & 0\,\big]
\end{tikzcd}
\end{equation}
in $\SMod_F$, we have the \textit{naturality property} 
\begin{equation*}
\label{eq:natural}
\textup{CTP}_{E } (\phi, \,{f_1^{\vee}}(\psi) ) \,=\, \textup{CTP}_{E'}(f_{2} (\phi), \psi),
\end{equation*}
for all $\phi$ in $\Sel\, M_{2 }$ and $\psi$ in $\Sel\, (M_{1}')^{\vee}$.
\end{itemize}

We will occasionally need to work with the explicit definition of the Cassels--Tate pairing in this paper, so we include this next. This pairing is modeled on the Weil pairing definition of the Cassels--Tate pairing for abelian varieties \cite[I.6.9]{Milne86}. That the various choices involved can be made, and that the result is independent of these choices, is justified in \cite[Section 3]{MS21}. 

\begin{defn}[\cite{MS21}, Definition 3.2] \label{def:CTP_small}
Take $E$ to be a short exact sequence in $\SMod_F$, as depicted in \eqref{eq:example_exact}. 
Take $\phi \in \Sel\, M_2$ and $\psi \in \Sel\, M_1^{\vee}$. To define $\text{CTP}_E(\phi, \psi)$ we make  `global' choices consisting of:
\begin{itemize}
\item  cocycles $\overline{\phi}$ and $\overline{\psi}$  representing $\phi$ and $\psi$ respectively,
\item a cochain $f\in C^1( G_{F}, M)$ satisfying
 $\overline{\phi} = \pi \circ f,$ and 
 \item a cochain $\epsilon$ in $C^2(G_F, (F^s)^{\times})$ so
\[d\epsilon \,=\, \iota^{-1}(df) \,\cup_{M_1}\, \overline{\psi}.\]
\end{itemize}   

We further make `local' choices consisting of:
\begin{itemize}
 \item cocycles $\overline{\phi}_{v, M}$ in $Z^1(G_v, M)$ for each place $v$ of $F$, so that  $\pi(\overline{\phi}_{v, M}) = \overline{\phi}_v$,
and so that $(\overline{\phi}_{v, M})_v$ represents a class in $\msW$. 
\end{itemize}
Then 
\[\gamma_v \,=\, \iota^{-1}(f_v - \overline{\phi}_{v, M}) \,\cup_{M_1} \,\overline{\psi}_v \,-\, \epsilon_v \]
lies in $Z^2\left(G_v, (F^s)^{\times}\right)$ for each $v$, and we have
\[\text{CTP}_E(\phi, \psi) = \sum_{v \text{ of }F} \inv_v(\gamma_v).\]
 \end{defn}

\section{Independence of the choice of embeddings} \label{sec:embedding_invariance}
In this section we address a technical point concerning the category $\SMod_F$ that is needed in the proof of Theorem \ref{thm:CT_shap}. In what follows, we refer to Appendix \ref{app:group_change} for the relevant background for various operations on cohomology.

\begin{notat}
Recall that we have fixed a field embedding $ \Fsep \hookrightarrow \Fsep_v$ for each place $v$ of $F$. Call this embedding $i_v$. This embedding determines the embedding of $G_v$ in $G_F$, and these embeddings are used in the construction of $\SMod_F$.

Suppose we have  a second set of embeddings $i'_v \colon \Fsep \hookrightarrow \Fsep_v$. Then there is $\tau_v$ in $G_F$ so $i'_v$ equals $i_v \circ \tau_v$, and the associated decomposition group is $\tau_v^{-1} G_v \tau_v$. Given a discrete $G_F$-module $M$, we have a natural isomorphism
\[\text{conj}_{\text{loc}}\colon\resprod_v H^1(G_v, M) \isoarrow \resprod_v H^1(\tau_v^{-1} G_v \tau_v, M)\]
given by conjugation at each place.

Taking $\SMod'_F$ to be the variant of $\SMod_F$ defined using the embeddings $i'_v$ instead of $i_v$, we can define an equivalence of categories $\SMod_F \to \SMod'_F$
by
\[(M, \msW) \mapsto (M, \text{conj}_{\text{loc}}(\msW)).\]
This respects duality and  Selmer groups. The following proposition shows it also respects the Cassels--Tate paring.
\end{notat}

\begin{prop} \label{prop:independence_of_embeddings}
Given $\SMod_F$ and $\SMod_F'$ as above, take $E$ to be a short exact sequence in $\SMod_F$, and take $E'$ to be the corresponding sequence in $\SMod_F'$. Then
\[\CTP_E = \CTP_{E'}.\]
\end{prop}

The proof relies on the following. Given a profinite group $G$, a closed subgroup $H$ of $G$,   a discrete $H$-module $M$, and $\tau\in G$, we can consider the $\tau H \tau^{-1}$-module $\tau M$, as detailed in Definition \ref{defn:exact_gc_functors}. As in Definition \ref{defn:conj_exp}, for each $n\geq 0$ we   have a conjugation operation 
\[\text{conj}_\tau^{H} \colon C^n(H, M) \to C^n(\tau H\tau^{-1},\, \tau M)\]
on cochains. In the case $G = H$, we have a natural isomorphism $M \cong \tau M$ given by the operation $c_\tau$ of Figure \ref{fig:nat_mor}, and composing this with   conjugation gives a chain map
\[\text{conj}_\tau^{G} \colon C^n(G, M) \to C^n(G, M).\]

This chain map is homotopic to the identity. Explicitly, we may follow \cite[Appendix 2]{Chung20} or \cite[4.5.5]{Neko06} to define a map
\begin{equation} \label{eq:explicit_homotopy}
h_{\tau} \colon C^{n+1}(G, M) \to C^n(G, M)
\end{equation}
by the formula
\[h_{\tau}(f)(\sigma_1, \dots, \sigma_n) = \sum_{r= 0}^{n} (-1)^r f(\sigma_1, \dots, \sigma_r, \,\tau,\, \tau^{-1} \sigma_{r+1}\tau, \dots, \tau^{-1}\sigma_n  \tau).\]
 On $C^0$  we take $h_{\tau}$ to be the zero map.

Given any $f$ in $C^n(G, M)$, this construction satisfies
\begin{equation} \label{eq:conjugation_homotopy_relation}
\text{conj}_\tau^{G} (f) - f \,=\, dh_{\tau}(f) + h_{\tau}(df).
\end{equation}

We may also verify that, given $G$-modules $M_1$ and $M_2$, integers $k,j\geq 0$, and 
\[f \in C^k(G, M_1) \quad\text{and}\quad g \in C^j(G, M_2),\]
 we have the identity
\begin{equation}
    \label{eq:homotopy_and_cup_product}
h_{\tau}(f \cup g) \,=\, h_{\tau}(f) \cup \tau g \,+\, (-1)^{k}\cdot f \cup h_{\tau}(g) \quad\text{in}\quad C^{k+j - 1}(G, M_1 \otimes M_2),
\end{equation}
where $\tau g$ is shorthand for $\text{conj}_\tau^{G} g$. These identities are most easily checked by considering the corresponding operations on homogeneous cochains (cf. Section \ref{sec:setup_appendix}).

\begin{proof}[Proof of Proposition \ref{prop:independence_of_embeddings}]
Take
\[E = \big[0 \to M_1 \xrightarrow{\,\,\iota\,\,} M \xrightarrow{\,\,\pi\,\,} M_2 \to 0\big]\]
to be a short exact sequence in $\SMod_F$, and choose $\phi \in \Sel M_2$ and $\psi \in \Sel M_1^{\vee}$. Take
\[(\overline{\psi},\, \overline{\phi},\, f, \,(\overline{\phi}_{v, M})_v, \,\epsilon)\]
to be a tuple for computing $\CTP_E(\phi, \psi)$ as in Definition \ref{def:CTP_small}. With $h_\tau$ as in \eqref{eq:explicit_homotopy} above, the proposition reduces to the claim that, for each place $v$ of $F$,
\[\iota^{-1}\left(f_v - \overline{\phi}_{v, M}\right) \cup \overline{\psi}_v \, -\, \epsilon_v\]
represents the same class in $H^2(G_v, \,(\Fsep_v)^{\times})$ as
\[\iota^{-1}\left((\tau f)_v - d h_{\tau}(f)_v - \overline{\phi}_{v, M}\right) \,\cup\, \,(\tau \overline{\psi})_v \, -\, (\tau\epsilon)_v.\]
To prove this claim, we start by using \eqref{eq:conjugation_homotopy_relation} to calculate that their difference equals
\[\iota^{-1}\left(f_v - \overline{\phi}_{v, M}\right) \cup dh_{\tau}(\overline{\psi})_v \, +\, \iota^{-1}\left(h_{\tau}(df)_v\right) \cup \left(\tau \overline{\psi}\right)_v \, -\, h_{\tau}(d\epsilon)_v \, -\, dh_{\tau}(\epsilon)_v.\]
From \eqref{eq:homotopy_and_cup_product}  we have
\[h_{\tau}(d\epsilon)_v \,=\, \iota^{-1}\left(h_{\tau}(df)_v\right) \cup \left(\tau \overline{\psi}\right)_v\, +\, \iota^{-1}(df)_v \cup h_{\tau}\left( \overline{\psi}_v\right),\]
so the difference of the expressions is equal to
\[\iota^{-1}\left(f_v - \overline{\phi}_{v, M}\right) \cup dh_{\tau}(\overline{\psi})_v \,-\, \iota^{-1}(df)_v \cup h_{\tau}\left( \overline{\psi}_v\right) \,-\, dh_{\tau}(\epsilon)_v,\]
which is the coboundary of 
\[-\iota^{-1}\left(f_v - \overline{\phi}_{M, v}\right) \cup h_{\tau}\left(\overline{\psi}\right)_v\,-\, h_{\tau}(\epsilon)_v.\]
The two expressions thus represent the same cohomology class, giving the claim.
\end{proof}

\section{Changing the global field}
\label{sec:fields}

In this section we consider the behavior of Selmer groups under change of field and prove Theorem \ref{thm:fields_main_intro}. As in the previous section, we will frequently refer to Appendix \ref{app:group_change} for the relevant background on cohomological operations related to field change.

Take $K/F$ to be an extension of global fields with $K$ contained in $F^s$. Given a $G_{K}$-module $M$ of order indivisible by the characteristic of $F$,  we consider the corresponding induced $G_{F}$-module 
\[\Ind_{K/F}M = \Z[G_F] \otimes_{\Z[G_K]} M.\] 
We wish to extend  the assignment $M\mapsto \Ind_{K/F}M$ to a functor  from $\SMod_{K}$ to  $\SMod_{F}$. For this we need to understand how local conditions subgroups behave under induction. This uses a local version of the Shapiro isomorphism, which we turn to now.

\subsection{Local and global Shapiro isomorphisms} \label{ssec:shap_loc} 
For $M$ as above, Shapiro's lemma gives a natural isomorphism
\begin{equation} \label{eq:global_shapiro_map}
\shap \colon H^1(G_{K}, \,M) \isoarrow  H^1(G_{F}, \Ind_{K/F} M)
\end{equation}
that we call the global Shapiro isomorphism. 

An equivalent  definition to the following local version of \eqref{eq:global_shapiro_map} can be found in \cite[3.1.2]{SkUr14}, with \eqref{eq:inverse_SU} appearing as \cite[(3.3)]{SkUr14}.
\begin{defn}
\label{defn:shap_loc}
Given a place $v$ of $F$ and a place $w$ of $K$ over $v$, the fields $F^s_v$ and $K_w^s$ are isomorphic. We have embeddings
\[i_v \colon \Fsep \to \Fsep_v \quad\text{and}\quad i_w \colon \Fsep \to K^s_w\]
corresponding to these places. Fixing some isomorphism between $F^s_v$  and $K^s_w$, there is some choice of $\tau_w \in G_F$ so that $i_v$ is identified with $i_w \circ \tau_w$. We then have
\[G_w = G_K \cap \tau_w^{-1} G_v \tau_w.\] 
The set  $\{\tau_w\}_{w\mid v}$ gives a set of representatives for the double coset  $G_v\backslash G_F/G_K$, and for any discrete $G_K$-module $M$ we can consider the isomorphism
\[\restr{\shap_{G_K}^{G_F}}{G_v} \colon \bigoplus_{w |v} H^k(G_w, M) \isoarrow  H^k(G_v, \Ind_{K/F} M)\]
of Definition \ref{defn:shap_GHC}. We abbreviate $\restr{\shap_{G_K}^{G_F}}{G_v}$ as  $\shap_{\text{loc}, \,v}$. In the case $k = 1$, we can apply \eqref{eq:sh_HGC_inertia} with $C_1 = I_v$ to show that this restricts to an isomorphism
\begin{equation}
    \label{eq:KF_unr}
\shap_{\text{loc}, \,v}^{\text{ur}} \colon \bigoplus_{w |v} H^1_{\text{ur}}(G_w, M) \isoarrow  H^1_{\text{ur}}(G_v, \Ind_{K/F} M).
\end{equation}
Combining the maps  $\shap_{\text{loc}, \,v}$ for every $v$  gives a topological isomorphism of restricted products
\begin{equation} \label{eq:sha_loc_version}
\shap_{\text{loc}}: \resprod_{w \text{ of } K} H^1(G_w, M) \isoarrow \resprod_{v \text{ of } F} H^1(G_v, \Ind_{K/F} M).
\end{equation}
\end{defn}

\subsection{The functor $\Ind_{K/F}$} \label{ssec:functor_ind}   We now define the sought functor $\Ind_{K/F}$.
\begin{defn}
\label{defn:Ind}
The  isomorphism \eqref{eq:sha_loc_version} sends compact open subgroups to compact open subgroups, so we may define a faithful additive functor
\[\Ind_{K/F} \colon \SMod_{K} \royarrow \SMod_{F}\]
by setting
\[(M, \msW)\,\mapsto\, \left(\Ind_{K/F} M, \,\shap_{\text{loc}}(\msW)\right).\]
 \end{defn}
 
A special case of \eqref{eq:sh_HGC_C_change} gives the commutative square
\begin{equation}
\label{eq:inverse_SU}
\begin{tikzcd}
H^k(G_{K}, M)\arrow{r}{\shap} \arrow[d, swap, "\bigoplus_w \res_{G_w}^{G_K}"] & H^k(G_F, \Ind_{K/F} M)  \arrow[d, "\res_{G_v}^{G_F}"]\\
\bigoplus_{w| v} H^k(G_w, M) \arrow{r}{\shap_{\text{loc},\, v}} & H^k(G_v, \Ind_{K/F} M),
\end{tikzcd}
\end{equation}
from which we see that the global Shapiro map gives an isomorphism
\[\Sel_KM \isoarrow \Sel_F(\Ind_{K/F}M).\]
Here the subscript indicates whether a functor is being considered on $\SMod_F$ or $\SMod_K$.
Consequently, the Shapiro isomorphism  defines a natural isomorphism
\[\shap \colon \Sel_K \natiso \Sel_F \circ \Ind_{K/F}.\]

We now consider how the functor $\Ind_{K/F}$ combines with the duality functors on $\Sel_F$ and $\Sel_K$. For this we require the following definition.

\begin{defn} \label{def:t_iso}
Given $(M, \msW)$ in $\SMod_{K}$, we consider the perfect $G_{F}$-equivariant bilinear pairing
\begin{equation}
\label{eq:Pt}
t_{M}\colon \Ind_{K/F} M \,\times\, \Ind_{K/F} M^{\vee} \xrightarrow{\quad} (F^s)^{\times}
\end{equation}
defined by
\[t_{M}\big( [\sigma_1] \otimes m_1,\, [\sigma_2] \otimes m_2 \big) \,=\,  \begin{cases} \sigma_1 \big\langle m_1, \, \sigma_1^{-1}\sigma_2 m_2\big\rangle &\text{ if } \sigma_1G_K= \sigma_2G_K, \\ 0&\text{ otherwise,}\end{cases}\]
where $\langle\,,\,\rangle$ is the evaluation pairing from $M \otimes M^{\vee}$ to $(F^s)^{\times}$. By tensor-hom adjunction, this pairing corresponds to a $G_{F}$-equivariant isomorphism
\begin{equation} \label{eq:where_t_is_defined}
 t_M: \Ind_{K/F} M^{\vee} \isoarrow (\Ind_{K/F} M)^{\vee}.
 \end{equation}
 \end{defn}
 
Per the following proposition,  these maps define a natural isomorphism
\begin{equation*}
t \colon \Ind_{K/F} \circ \vee_K \natiso \vee_F \circ \Ind_{K/F}
\end{equation*}
on $\SMod_K$.

 \begin{prop} \label{prop:duality_functor}
Let $(M,\mathscr{W})$ be an object in $\SMod_{K}$. We then have
\[t\left(\shap_{\textup{loc}}(\mathscr{W}^\perp)\right)=\shap_\textup{loc}(\mathscr{W})^\perp.\]
In particular, $t$ gives an isomorphism 
\[t\colon \Ind_{K/F} ~(M,\mathscr{W})^\vee \isoarrow \left(\Ind_{K/F}(M,\mathscr{W})\right)^\vee.\]
\end{prop}
We will prove this proposition in Section \ref{ssec:field_proofs}. From this proposition, the tuple $(\Ind_{K/F}, \, t)$ forms a morphism
\[(\Ind_{K/F}, \, t) \colon (\SMod_{K}, \,\vee_K) \royarrow (\SMod_{F}, \,\vee_F)\]
of additive categories with duality in the sense of \cite[Definition 1.1.8]{MR2181829}. 

The functor $\Ind_{K/F}$ is exact, and we may compare the Cassels--Tate pairing of a short exact sequence $E$ in $\SMod_K$ with the pairing for $\Ind_{K/F} E$. The following theorem, which formalizes Theorem \ref{thm:fields_main_intro}, shows that these pairings are essentially equivalent.
\begin{thm}
\label{thm:fields_main}
\label{thm:CT_shap}
 Choose any short exact sequence 
\[E = \left[0 \to M_1 \xrightarrow{\,\,\iota\,\,} M \xrightarrow{\,\,\pi\,\,} M_2 \to 0\right]\]
in $\SMod_K$.  Then, for any $\phi$ in $\Sel_K\, M_2$ and $\psi$ in $\Sel_K\, M_1^{\vee}$, we have 
 \begin{equation}
\label{eq:CT_shap}
\CTP_{\Ind_{K/F} E}\big(\shap(\phi), \,\,t(\shap(\psi))\big) = \CTP_E(\phi, \, \psi).
\end{equation}
\end{thm}

As with Proposition \ref{prop:duality_functor}, we will prove this in Section \ref{ssec:field_proofs}.

\begin{ex}
\label{ex:res_of_scalars}
Given $K/F$ as above and an abelian variety $A/K$, we may consider the restriction of scalars $R_{K/F}\,A$, which is an abelian variety over $F$ defined as in e.g. \cite{Milne72}. We have a canonical isomorphism of discrete $G_F$-modules
\[\Z[G_F] \otimes_{\Z[G_K]} A(F^s) \,\cong\, R_{K/F}\,A(F^s).\]
Given an integer $n$ coprime to the characteristic of $F$, we have an isomorphism 
\begin{equation}
\label{eq:classical_Ind}
\left(R_{K/F}\,A[n],  \, \msW_n^F\right)  \, \cong \, \Ind_{K/F}\left(A[n], \,\msW_n^K\right)
\end{equation}
in $\SMod_F$, where $\msW_n^F$ and $\msW_n^K$ denote the local conditions defining the $n$-Selmer groups of $R_{K/F}\,A/F$ and $A/K$ respectively.  The Shapiro isomorphism gives an identification
\[\Sel^n (A/K) \,\cong\, \Sel_F \left( \Ind_{K/F} A[n]\right)\,\cong\, \Sel^n (R_{K/F} \,A/F).\]
In this language, Theorem \ref{thm:CT_shap}  says that the Cassels--Tate pairing of $A$ over $K$ is identified with the Cassels--Tate pairing of $R_{K/F}\,A$ over $F$.
\end{ex}
\begin{rmk}
\label{rmk:infinite_F_change}
If we instead define $\SMod_F$ and $\SMod_K$ using the weaker notion of suitable local conditions given in \cite[Definition 4.16]{MS21}, then Definition \ref{defn:Ind}  still defines a functor 
\[\Ind_{K/F}\colon \SMod_K \to \SMod_F,\] 
and Proposition \ref{prop:duality_functor} and Theorem \ref{thm:fields_main} still apply.

In addition, we may use the explicit constructions of Section \ref{sec:explicit_cochain_formulae} to extend $\Ind_{K/F}$ to a functor from $\SMod_{K, \ell}^{\infty}$ to $\SMod_{F, \ell}^{\infty}$ for each $\ell$ not equal to the characteristic of $F$, and to extend the natural isomorphism $\shap_{K/F}$ in this context. Here $\SMod_{K, \ell}^{\infty}$ and $\SMod_{F, \ell}^{\infty}$   are the categories of potentially-infinite Galois $\Z_{\ell}$-modules decorated with local conditions introduced in \cite[Section 4.2]{MS21}. We still have Proposition \ref{prop:duality_functor} and Theorem \ref{thm:fields_main} in this context, as can be proved from \cite[Lemma 4.9]{MS21} and the definition of the Cassels--Tate pairing in these categories.
\end{rmk}

Before turning to the proofs of Proposition \ref{prop:duality_functor} and Theorem \ref{thm:fields_main}, we give some basic consequences of these results.

\subsection{The effect of corestriction and restriction}
By \eqref{eq:cor_from_sh}, restriction and corestriction can  be phrased in terms of the Shapiro isomorphism, so Theorem \ref{thm:fields_main} can be used to   understand their effect on   the Cassels--Tate pairing. We begin with a definition. Here and below, we refer to Figure \ref{fig:nat_mor} for the definition of the maps $i$, $\nu$, $i_1$ and $\nu_1$. 

\begin{defn}
For each $w$ dividing $v$, we define natural homomorphisms
\begin{align*}
&\cores_{w|v} \,\colon H^k(G_w, M) \to H^k(G_v, M)\quad\text{and}\\
&\res_{w|v}\, \colon H^k(G_v, M) \to H^k(G_w, M)
\end{align*}
by setting
\[\cores_{w|v} \,=\, \cores^{\tau_w G_w \tau_w^{-1}}_{G_v} \,\circ\, \tau_w \quad\text{ and }\quad \res_{w|v} \,=\, \tau_w^{-1} \,\circ\,\res^{G_v}_{\tau_w G_w \tau_w^{-1}}.\]
By \eqref{eq:shHGU_def}, the $w$ component of $\shap_{\text{loc},\, v}$ is  the composition of the homomorphisms
\[H^k(G_w, M) \xrightarrow{\,\,i_1\,\,}  H^k\left(G_w, \,\Ind_{K/F} M\right) \xrightarrow{\,\, \cores_{w|v}\,\,} H^k\left(G_v, \,\Ind_{K/F} M\right),\]
and the $w$ component of the inverse is the composition
\[H^k\left(G_v, \,\Ind_{K/F} M\right) \xrightarrow{\,\, \res_{w|v}\,\,}  H^k\left(G_w, \,\Ind_{K/F} M\right) \xrightarrow{\,\,\nu_1\,\,} H^k(G_w, M).\]
If $M$ is a $G_F$-module we also have a commutative diagram
\begin{equation}
\label{eq:cor_from_shloc}
\begin{tikzcd}[column sep = large]
\bigoplus_{w|v} H^k(G_w, M) \arrow[dr, "\shap_{\text{loc}, v}"] \arrow[r,  "\sum_{w|v} \cores_{w|v}"] &H^k(G_v, M) \\
 H^k(G_v,  M) \arrow[u,"\bigoplus_{w|v} \res_{w|v}"] \arrow[r, "i"]  & H^k(G_v, \Ind_G^H M) \arrow[u,"\nu"],
\end{tikzcd}
\end{equation}
as can be seen by verifying the result for $k = 0$.
\end{defn}

In what follows we abbreviate $\cores_{G_{F}}^{G_{K}}$ as $\cores$ and $\res_{G_{K}}^{G_{F}}$ as $\res$. 

\begin{cor}
\label{cor:cor}
 Given exact sequences
\begin{alignat*}{3}
& E_F\,=\,&&\left[0 \rightarrow (M_1, \msW_{1F}) \xrightarrow{\,\,\iota\,\,} (M, \msW_F) \xrightarrow{\,\,\pi\,\,} (M_2, \msW_{2F}) \rightarrow 0\right]\quad&&\text{in } \SMod_{F} \quad\text{and}\\
 &E_K\,=\,&&\left[0 \rightarrow (M_1, \msW_{1K}) \xrightarrow{\,\,\iota\,\,} (M, \msW_K) \xrightarrow{\,\,\pi\,\,} (M_2, \msW_{2K}) \rightarrow 0\right]\quad&&\text{in } \SMod_{K},
\end{alignat*}
associated to the same exact sequence of $G_{F}$-modules, we have the following:
\begin{enumerate}
\item Suppose 
\[  \sum_{v\text{ of }F}\sum_{w\mid v}\textup{cor}_{w\mid v}(\phi_w)\in  \msW_F\quad\text{for all}\quad (\phi_w)_w \in \msW_K.\]
Then, for any $\phi \in \Sel_K\, M_2$ and $\psi \in \Sel_F\, M_1^{\vee}$, we have
\[\CTP_{E_F}\left(\cores \,\phi, \,\psi\right)\, =\, \CTP_{E_K}\left(\phi, \,\res\, \psi\right).\]
\item Suppose 
\[\sum_{v\text{ of } F}\sum_{w\mid v}\textup{res}_{w\mid v}(\phi_v)\in \msW_K \quad\text{for all}\quad (\phi_v)_v \in \msW_F.\]
Then, for any $\phi \in \Sel_F\, M_2$ and $\psi \in \Sel_K\, M_1^{\vee}$, we have
\[\CTP_{E_F}\left(\phi,\, \cores\, \psi\right)\, =\, \CTP_{E_K}\left(\res \,\phi, \,\psi\right).\]
\item Suppose the conditions of the previous two parts both hold. Then, for any $\phi$ in  $\Sel_F\, M_2$ and $\psi$ in $\Sel_F\, M_1^{\vee}$, we have
\[ \CTP_{E_K}\left(\res\,\phi,\, \res\,\psi\right) = [K : F] \cdot \CTP_{E_F}(\phi, \,\psi). \]
\end{enumerate}
\end{cor}
\begin{proof}
From \eqref{eq:cor_from_sh}, we have
\[\cores = \nu\circ \shap \quad\text{and}\quad \res = \shap^{-1} \circ i.\]
For the first part, we start with the commutative diagram of $G_F$-modules
\begin{equation}
\label{eq:cor_nu_exact}
\begin{tikzcd}
 0 \arrow{r} & \Ind_{K/F} M_1 \arrow{d}{\nu} \arrow{r}{\textup{Ind}(\iota)}& \Ind_{K/F} M  \arrow{d}{ \nu}    \arrow{r}{\textup{Ind}(\pi)}&\Ind_{K/F} M_2 \arrow{d}{ \nu}    \arrow{r} & 0\\
0 \arrow{r} & M_1 \arrow{r}{\iota} & M \arrow{r}{\pi} & M_2 \arrow{r} & 0
\end{tikzcd}
\end{equation}
with exact rows. From \eqref{eq:cor_from_shloc} and our assumptions, this induces a commutative diagram in  $\SMod_{F}$ with top row $\Ind_{K/F} E_K$ and   bottom row  $E_F$. We also note that the diagram
\[\begin{tikzcd}
M_1^{\vee} \arrow[r, "\nu^{\vee}"] \arrow[rr, bend right = 14, "i" '] &(\Ind_{K/F} M_1)^{\vee}\arrow[r, "t^{-1}"] & \Ind_{K/F} M_1^{\vee}
\end{tikzcd}\]
is commutative. The part then follows from Theorem \ref{thm:CT_shap}, \eqref{eq:cor_from_sh}, and naturality applied to \eqref{eq:cor_nu_exact}. The proof for the second part is similar.

Under the conditions of the third part, we find that $\res\, \phi$ is in $\Sel_{K} \,M_2$. The conditions of the first part also hold, so
\[\CTP_{E_K}\left(\res\,\phi,\, \res\,\psi\right) \,=\, \CTP_{E_F}\left(\cores\, \circ\res\,\phi,\,\psi\right) \,=\,   \CTP_{E_F}\left([K: F] \cdot \phi, \,\psi\right), \]
giving the claim.
\end{proof}
\begin{rmk}
\label{rmk:Yu}
The first two parts of this corollary were proved by Fisher for the Cassels--Tate pairing of an elliptic curve over a number field \cite{Fish03}. For more general abelian varieties, a sketched proof for these identities appears in \cite[Theorem 8]{Yu04}. We mention that the latter work relies on a cochain identity that does not hold for the standard resolutions used to define the Cassels--Tate pairing  \cite[(6)]{Yu04}. The third part of the corollary extends a result for abelian varieties due to   \v{C}esnavi\v{c}ius \cite[Lemma 3.4]{MR3860395}.

If $[K \colon F]$ does not divide the order of $M$, Corollary \ref{cor:cor} (3)  shows that the Cassels--Tate pairing for $E$ over $F$ can be recovered from the Cassels--Tate pairing over $K$. In particular, if $M$ has order a power of $\ell$, we can compute the Cassels--Tate pairing for $E$ by computing it over any field $K$ for which the image of $G_{K}$ in $\textup{Aut}(M)$  is an $\ell$-group. 
\end{rmk}

\subsection{The effect of conjugation}
Given $K/F$ as above, and given $\tau$ in $G_F$, we can define a functor
\[\tau \colon \SMod_K  \to \SMod_{\,\tau K}\]
by taking
\[\tau(M, \msW) = (\tau M, \,\tau \msW).\]
 Conjugation by $\tau$ then defines a natural isomorphism
\[\tau \colon \Sel_K \circ \tau \natiso \Sel_{\tau K}.\]

Given  $M$ in $\SMod_K$, the map $\rho_{\tau}$ defined in Figure \ref{fig:nat_mor} gives a natural isomorphism
\[\rho_{\tau} \colon \Ind_{K/F} M \isoarrow \Ind_{\tau K/F}\, \tau M\]
in the category $\SMod_F$. By \eqref{eq:rho_tau_square}, we have a commutative square
\begin{equation}
\label{eq:rho_tau_used}
\begin{tikzcd}
\Sel_K(M) \arrow{r}{\tau} \arrow{d}{\shap_{K/F}} & \Sel_{\tau K} (\tau M )\arrow{d}{\shap_{\tau K/F}} \\
\Sel_F\left( \Ind_{K/F} M \right) \arrow{r}{\rho_{\tau}} & \Sel_F\left( \Ind_{\tau K/F} \,\tau M \right).
\end{tikzcd}
\end{equation}

Naturality of the Cassels--Tate pairing and Theorem \ref{thm:fields_main}  gives the following, which can also be proven directly from the definition of the Cassels--Tate pairing.   
\begin{prop}
\label{prop:conjugation}
Given a short exact sequence $E$ in $\SMod_K$, and given $\tau$ in $G_F$, we have
\[ \CTP_{\tau E}(\tau \phi, \, \tau \psi) \,=\,\CTP_{E}(\phi, \,\psi).\]
\end{prop}
 
We consider one special case of this proposition. Suppose that $K/F$ is Galois, and suppose we have an exact sequence of $G_F$-modules
\[0 \rightarrow M_1 \xrightarrow{\,\,\iota\,\,} M \xrightarrow{\,\,\pi\,\,} M_2 \rightarrow 0.\]
Suppose also that we augment this to an exact sequence  
\begin{align*}
& E\,=\,&&\left[0 \rightarrow \left(\text{Res}_{G_K} \,M_1, \,\msW_1\right) \xrightarrow{\,\,\iota\,\,} \left(\text{Res}_{G_K} \,M, \,\msW\right) \xrightarrow{\,\,\pi\,\,} \left(\text{Res}_{G_K} \,M_2,\, \msW_2\right) \rightarrow 0\right]
\end{align*}
in $\SMod_K$ so that 
\begin{equation}
\tau(\mathscr{W})=\mathscr{W}~~\quad \textup{for all} \quad\tau \in \Gal(K/F).
\end{equation}
Under this assumption, the Selmer group for $M$, $M_1$, $M_2$, and their duals are all $\Gal(K/F)$-modules under the conjugation action. The above proposition gives
\begin{equation}
\label{eq:conjugation_special}
\CTP_{E}(\tau \phi, \, \tau \psi) \,=\,\CTP_{E}(\phi, \,\psi)
\end{equation}
for all $\tau$ in $\Gal(K/F)$, $\phi$ in $\Sel_{K} M_2$, and $\psi$ in $\Sel_{K} M_1^{\vee}$.

\subsection{Proofs of Proposition \ref{prop:duality_functor} and Theorem \ref{thm:fields_main}}
\label{ssec:field_proofs}
In Section \ref{sec:basic_operations}, we defined corestriction using injective resolutions. This contrasts with our definition of the Cassels--Tate pairing, which is defined explicitly by manipulating inhomogeneous cochains.

To prove Theorem \ref{thm:fields_main}, we need these definitions to be more compatible. For the work below, we do this by  redefining corestriction in terms of inhomogeneous cochains. A potentially more elegant approach would be to redefine the Cassels--Tate pairing using injective resolutions. We hope to return to this alternative in future work.

In the proofs of Proposition \ref{prop:duality_functor} and Theorem \ref{thm:CT_shap} we  will use the following class field theoretic lemma.

\begin{lem}
\label{lem:inv_cor}
Given $v$ a place of $F$ and $w$ a place of $K$ dividing $v$, we have a commutative diagram
\begin{equation}
\label{eq:inv_cor}
\begin{tikzcd}
H^2(G_w, \,\Fvstar) \arrow[rr, "{ \cores_{w|v}}"] \arrow[dr, "\inv_{w}" '] && H^2(G_v, \,\Fvstar)\arrow[dl, "\inv_{v}"]\\
& \QQ/\Z &.
\end{tikzcd}
\end{equation}
\end{lem}
\begin{proof}
For $v$ nonarchimedean  this is \cite[(7.1.4)]{Neuk08}. If $v$ is archimedean, it can be verified by separate consideration of the two cases $K_w = \mathbb{C}$ and $K_w = \mathbb{R}$.
\end{proof}

\begin{proof}[Proof of Proposition \ref{prop:duality_functor}]
Given a place $v$ of $F$, and given $\phi_w$ in $H^1(G_w, M)$ and $\psi_w$ in $H^1(G_w, M^{\vee})$ for each place $w$ of $K$ dividing $v$, it suffices to show that
\begin{equation}
    \label{eq:du_func_goal}
\inv_v\left(\shap_{\text{loc}, \, v}\big((\phi_w)_w\big) \,\cup_{t_M}\, \shap_{\text{loc}, \, v}\big((\psi_w)_w\big)  \right)\, =\, \sum_{w | v} \inv_w(\phi_w \cup \psi_w).
\end{equation}
For each $w$, we define $\tau_w$ in $G_F$ as in Definition \ref{defn:shap_loc}. We refer to Figure \ref{fig:nat_mor} for the definition of the maps $\nu$ and $\nu_1$. The pairing $t_{M}$ of \eqref{eq:Pt} may be expressed as a composition
\[\Ind_{K/F} M \times \Ind_{K/F} M^{\vee}  \xrightarrow{\,\,P\,\,}\Ind_{K/F}(M \otimes M^{\vee}) \xrightarrow{\,\,\Ind_{K/F}B\,\,}  \Ind_{K/F} (\Fsep)^{\times} \xrightarrow{\,\,\nu\,\,} (\Fsep)^{\times},\]
where $B$ is the evaluation pairing on $M \otimes M^{\vee}$ and where $P$ is the pairing defined in Lemma \ref{lem:gc_appendix} (2). Taking $B'$ to be the pairing $\Ind_{K/F}B \circ P$, we have the identity 
 \begin{equation}
     \label{eq:coch_cup_nu}
 \nu_1 \circ \res_{G_w} \circ \tau_w^{-1}\left(a \cup_{B'} b\right) \,=\, \nu_1 \circ \res_{G_w} \circ \tau_w^{-1}(a) \,\cup_B\, \nu_1 \circ \res_{G_w} \circ \tau_w^{-1}( b),
\end{equation}
for any $a$ in $C^1(G_v, \Ind_{K/F} M)$ and $b$ in $C^1(G_v, \Ind_{K/F} M^{\vee})$, and for any $w$ dividing $v$. This is readily verified  at the level of cochains. 

To ease notation, we take
\begin{equation}
\label{eq:nrtw}
\text{nrt}_w = \nu_1 \circ \res_{G_w} \circ \tau_w^{-1}.
\end{equation}
Note that this is the $w$-component of $\shap_{\text{loc},\,v}^{-1}$.

From \eqref{lem:inv_cor} and \eqref{eq:cor_from_shloc}, we have a commutative triangle
\begin{equation}
\label{eq:inv_cor_2}
\begin{tikzcd}
H^2\left(G_v, \,\Ind_{K/F} \,(\Fsep_v)^{\times}\right) \arrow[rr, "\oplus_{w |v}\,\text{nrt}_w"] \arrow[dr, "\inv_v\,\circ\, \nu" '] && \bigoplus_{w |v}H^2(G_w, \,(\Fsep_v)^{\times})\arrow[dl, "\sum_{w | v} \inv_{w}"]\\
& \QQ/\Z. &
\end{tikzcd}
\end{equation}
The left hand side of \eqref{eq:du_func_goal} thus equals
\begin{alignat*}{2}
&\sum_{w | v}\inv_w \circ\text{nrt}_w \left(\shap_{\text{loc}, \, v}\big((\phi_w)_w\big) \,\cup_{B'}\,  \shap_{\text{loc}, \, v}\big((\psi_w)_w\big) \right) && \\
&\quad = \sum_{w | v}\inv_w\left(\text{nrt}_w\circ \shap_{\text{loc}, \, v}\big((\phi_w)_w\big) \,\cup \, \text{nrt}_w\circ \shap_{\text{loc}, \, v}\big((\psi_w)_w\big) \right))\quad &&\text{by \eqref{eq:coch_cup_nu}}\\
&\quad =  \sum_{w | v} \inv_w(\phi_w \cup \psi_w). &&
\end{alignat*}
This establishes \eqref{eq:du_func_goal}, giving the proposition.
\end{proof}

\begin{proof}[Proof of Theorem \ref{thm:CT_shap}]
Take $G = G_F$ and $H = G_K$, and choose a transversal $T$   as in Definition \ref{defn:explicit_shap}. For each place $w$ of $K$,  define $\tau_w$ as in Definition \ref{defn:shap_loc}. Applying Proposition \ref{prop:independence_of_embeddings} as necessary,  we may  assume  the $\tau_w$ all lie in the image of  $T$.

Choose cocycles $\overline{\phi}$ and $\overline{\psi}$ representing $\phi$ and $\psi$, and choose $f$ in $C^1(H, M)$ projecting to $\overline{\phi}$. Choose $\epsilon$ in $C^2(G, (F^s)^\times)$ satisfying $d\epsilon = \iota^{-1}(df) \cup_{M_1} \overline{\psi}$.

We take
\[f' =  {}^T\shap^H_G (f),\quad \epsilon' =   {}^T\shap^H_G (\epsilon),\quad \overline{\phi}' =  {}^T\shap^H_G (\overline{\phi}),\quad \overline{\psi}' =  {}^T\shap^H_G (\overline{\psi}).\]
Lemma \ref{lem:gc_appendix} then implies $\epsilon' \in C^2\left(G_F,\, \Ind_{K/F} (\Fsep)^{\times}\right)$ satisfies
\[d\epsilon' = (\Ind_{K/F}\,\iota)^{-1}(df')  \cup_{B'} \overline{\psi'},\]
where the pairing $B'$ is defined as in the proof of Proposition \ref{prop:duality_functor} above.

Choose a tuple of cocycles $(\overline{\phi}_{v,\, \text{Ind } M})_v$ representing a class in $\shap_{\text{loc}}(\msW)$ so 
$\overline{\phi}_{v,\, \text{Ind } M}$ projects to the cocycle $\overline{\phi}'_v$ for each $v$ of $F$. Defining $\text{nrt}_w$ as in \eqref{eq:nrtw}, we then take
\[\overline{\phi}_{w, M} = \text{nrt}_w(\overline{\phi}_{v,\, \text{Ind } M}).\]
By part three of Lemma \ref{lem:gc_appendix}  we have 
$\pi\left(\overline{\phi}_{w, M}\right) = \overline{\phi}_w$
for each $v$. In the notation of Definition \ref{def:CTP_small}, we may compute $\CTP_E(\phi, \psi)$ from the tuple
\[(\overline{\psi},\, \overline{\phi},\, f, \,(\overline{\phi}_{w, M})_v, \,\epsilon),\]
and we may compute $\CTP_{\Ind_{K/F}\, E}(\shap_{K/F} \,\phi, \,\shap_{K/F}\, \psi)$ from the tuple
\[\left(\overline{\psi}',\, \overline{\phi}',\, f',\, \left(\overline{\phi}_{v,\, \text{Ind } M}\right)_v,\, \nu \circ \epsilon'\right).\]
Here $\nu$ is  defined as in  Figure \ref{fig:nat_mor}.  For each place $v$ of $F$ we then have  
\begin{alignat*}{2}
    &\inv_v\left(\left(f'_v - \overline{\phi}_{v,\, \text{Ind } M}\right) \cup_{t} \overline{\psi}_v  \,-\, \nu \circ \epsilon'_v\right) &&\\[2pt]
    &\quad= \inv_v\circ\nu\left(\left(f'_v - \overline{\phi}_{v,\, \text{Ind } M}\right) \cup_{B'} \overline{\psi}_v  \,-\, \epsilon'_v\right)&&\\[2pt]
   &\quad=\sum_{w|v} \inv_w \circ \text{nrt}_w\left(\left(f'_v - \overline{\phi}_{v,\, \text{Ind } M}\right) \cup_{B'} \overline{\psi}_v  \,-\, \epsilon'_v\right)\,\,\,\, &&\text{by \eqref{eq:inv_cor_2}}\\
&\quad = \sum_{w|v} \inv_w\left( \text{nrt}_w\left(f'_v - \overline{\phi}_{v,\, \text{Ind } M}\right) \,\cup_{M_1}\, \text{nrt}_w(\overline{\psi}'_v) \,-\, \text{nrt}_w(\epsilon'_v)\right)\,\,\,\,&&\text{by  \eqref{eq:coch_cup_nu}}  \\
 &\quad = \sum_{w|v} \inv_w\left( \left(f_w -  \overline{\phi}_{w,\, M}\right) \,\cup_{M_1}\, \overline{\psi}_w\, -\, \epsilon_w\right)&&\text{by Lemma \ref{lem:gc_appendix}},
\end{alignat*}
where we have suppressed the appearances of $(\Ind_{K/F}\, \iota)^{-1}$ and $\iota^{-1}$. Summing over all places $v$ gives the result.
\end{proof}

\section{Class groups and roots of unity}
\label{sec:examples} \label{sec:class_groups}
We now shift our focus to applications of our theory of the Cassels--Tate pairing. Our first goal is to prove Propositions \ref{prop:cg_sym_intro} and \ref{prop:cg_alt_intro} on the class groups of global fields with many roots of unity. To start, we follow Flach \cite{Flach90} to translate the reciprocity pairing into a Cassels--Tate pairing.

\subsection{The reciprocity pairing as a Cassels--Tate pairing}
\label{ssec:fund_class_groups}
\begin{notat}
Take $F$ to be a global field, and write $\mcO_F$ for the ring of integers of $F$. For each positive integer $m$ indivisible by the characteristic of $F$, take
\[\msW_m^{\text{ur}} = \prod_v H^1_{\text{ur}}\left(G_v, \tfrac{1}{m}\Z/\Z\right),\]
with the product being over all places of $F$.  We  have
\[\textup{Sel}\left( \tfrac{1}{m}\Z/\Z,\msW_m^{\text{ur}}\right) \,=\, \Hom\left(\Gal(H_F/F), \tfrac{1}{m}\Z/\Z\right) \,\stackrel{\textup{def}}{=}\, \Cl^*F[m],\]
where $H_F$ is the Hilbert class field of $F$. 

To ease notation in what follows, set
\[\quad\quad \quad \quad\quad \quad\underline{\tfrac{1}{m}\Z/\Z}=(\tfrac{1}{m}\Z/\Z, \msW_m^{\text{ur}})\quad\quad \quad\textup{in }\SMod_F.\]
Similarly, we will denote by $\underline{\mu_m}$ the dual object   $(\mu_m, \msW_m^{\text{ur}, \perp})$.\end{notat}

The  Selmer group   of $\underline{\mu_m}$   turns out to be isomorphic to the $m$-Selmer group of $F$  \cite[Definition 5.2.4]{Cohen00}, whose definition we now recall.

\begin{defn} 
The \emph{$m$-Selmer group} of $F$ is defined to be the group
\[\textup{Sel}^m F=\big\{ \alpha \in F^{\times}/(F^{\times})^m \,:\,\, v(\alpha)  \equiv 0~~\,\textup{mod }m\, \textup{ for all nonarch. places }v \textup{ of }F\}.\]
Here we are thinking of a nonarchimedean place $v$ in terms of the corresponding normalized valuation.
Note that given $\alpha\in F^\times$ representing an element in $ \textup{Sel}^m F$, we have 
$(\alpha)=I^m$
for some fractional ideal $I$ of $F$. This defines a homomorphism
\[\pi_{\text{Cl}, m}:\textup{Sel}^m F \to \Cl F[m]\]
taking $\alpha$ to  the class of $I$. The $m$-Selmer group of $F$ then sits in the short exact sequence
\begin{equation}
\label{eq:Sel_NF}
1 \to \mcO_F^{\times}/(\mcO_F^{\times})^m \xrightarrow{\quad} \Sel^m F \xrightarrow{\,\,\pi_{\text{Cl},m}\,\,} \Cl F[m] \to 1,
\end{equation}
the first  map induced by the natural inclusion of $\mathcal{O}_{F}^{\times}$ into $F^\times$. 
\end{defn}

\begin{notat}
Write $\delta_m$ for the connecting map 
\[\delta_m\colon F^{\times}/(F^{\times})^m \isoarrow H^1(G_F, \mu_m)\]
associated to the sequence 
\[1 \to \mu_m \to \Fstar \to \Fstar \to 1.\]
 For a place $v$, we write $\delta_{m,v}$ for the corresponding local isomorphism
\[F_v^{\times}/F_v^{\times m} \isoarrow H^1(G_v,\mu_m).\]
\end{notat}

For a nonarchimedean place $v$ of $F$ the local Tate pairing 
\[H^1(F_v,\tfrac{1}{m}\mathbb{Z}/\mathbb{Z})\otimes H^1(F_v,\mu_m)\to \mathbb{Q}/\mathbb{Z}\]
is given by 
\begin{equation}\label{eq:local_artin_map_pairing}
(\chi,\delta_{m,v}(x))\longmapsto \chi(\textup{rec}_{F_v}(x)),  
\end{equation}
where  $\textup{rec}_{F_v}:F_v^\times \to G_v^\textup{ab}$
 is the local Artin map.
This follows from \cite[Theorem 7.2.13]{Neuk08}  and antisymmetry of the cup-product. 

\begin{lem} \label{lem:relevance_of_m_sel}
The map $\delta_m$ gives an isomorphism
\[ \Sel^m F \isoarrow \textup{Sel}~\underline{\mu_m}.\]
\end{lem}

\begin{proof}
For each nonarchimedean place $v$, it follows from \eqref{eq:local_artin_map_pairing} and \cite[ Proposition XIII.13]{Serre79} that the orthogonal complement of $H^1 (G_v/I_v,\tfrac{1}{m}\mathbb{Z}/\mathbb{Z})$ under the local Tate pairing is  equal to 
\begin{equation*} \label{eq:orthog_of_un_is_units}
\delta_{m,v}\left(\mathcal{O}_{F_v}^{\times}/\mathcal{O}_{F_v}^{\times m}\right).
\end{equation*}
Thus the Selmer group $\textup{Sel}~\underline{\mu_m}$  corresponds under $\delta_m^{-1}$ to the group
\[\ker\Bigg(F^\times/F^{\times m}\to \prod_{\substack{v~\textup{place of }F\\ v\textup{ nonarch.}}}(F_v^{\times}/F_v^{\times m})/\mathcal{O}_{F_v}^{\times}\Bigg).\]
This kernel is precisely the $m$-Selmer group of $F$.  
\end{proof}

Given positive integers $a, b$, we now consider the exact sequence
\begin{equation}
\label{eq:Eab_def}
 E(a, b)\, =\, \left[0 \to \underline{\tfrac{1}{a}\Z/\Z} \xrightarrow{\quad} \underline{\tfrac{1}{ab}\Z/\Z}  \xrightarrow{\,\,\cdot a\,\,} \underline{\tfrac{1}{b}\Z/\Z} \to 0\right]
\end{equation}
in the category $\SMod_F$. The associated Cassels--Tate pairing gives a pairing
\[\CTP_{E(a, b)} \colon \textup{Cl}^*F[b]\otimes \textup{Sel}~ \underline{\mu_a}  \to \mathbb{Q}/\mathbb{Z}.\]
This pairing is related to the reciprocity pairing \eqref{eq:reciprocity_pairing_intro} as follows. 

\begin{prop} \label{prop:agrees_with_artin_map}
Choose $\phi\in \textup{Cl}^* F[b]$ and $\psi\in \textup{Sel}~\underline{\mu_a} $. Write $\alpha=\delta_a^{-1}(\psi)$ for the element of $\textup{Sel}^a F$ corresponding to $\psi$ via Lemma \ref{lem:relevance_of_m_sel}. Then we have
\begin{equation}
\label{eq:EabRP}
\CTP_{E(a, b)}(\phi,\psi)=-\textup{RP}(\phi,\,\pi_{\textup{Cl}, a}(\alpha)).
\end{equation}
\end{prop}

\begin{proof}
View $\alpha$ as an element in $F^\times$ by picking a representative for its class in $\textup{Sel}^m F$.
For each nonarchimedean place $v$, choose $\alpha_v, \beta_v$ in $F_v^{\times}$ so $\alpha_v$ has trivial valuation and so
$\alpha = \alpha_v \beta_v^a.$
From \cite[Definition 3.8 and Lemma 3.9]{MS21}, we find
\begin{align*}
\CTP_{E(a, b)}(\phi, \,\psi) = \,&-\sum_{\substack{v \text{ nonarch.}}} \inv_v\left(\text{res}_{G_v}(\phi) \cup \delta_{ab,\, v}(\alpha/\alpha_v)\right) \\
= \,&-\sum_{\substack{v \text{ nonarch.}}} \inv_v\left(\text{res}_{G_v}(\phi)\cup \delta_{b,\, v}(\beta_v)\right).
\end{align*}
Taking $\text{rec}_{F_v}\colon F_v^{\times} \to G_v^{\text{ab}}$ to be the local Artin map for each nonarchimedean $v$, we have
\[\text{res}_{G_v}(\phi)\cup \delta_{b, v}(\beta_v) \,=\, \phi(\textup{rec}_{F_v}(x)).\] 
This follows from \cite[Theorem 7.2.13]{Neuk08}  and antisymmetry of the cup-product. Summing these terms gives \eqref{eq:EabRP}.
\end{proof}

\subsection{Self-duality in the presence of roots of unity}

Fix, once and for all, an isomorphism
\begin{equation} \label{eq:root_of_unity_map_g}
g \colon \QQ/\Z \isoarrow \bigcup_{n \ge 1} \mu_n
\end{equation}
between $\QQ/\Z$ and the set of roots of unity in $F^s$.

Suppose $F$ contains $\mu_m$ for some integer $m\geq 1$. Then the restriction of $g$ to $\tfrac{1}{m}\mathbb{Z}/\mathbb{Z}$ is $G_F$-equivariant. Furthermore, we have 
\begin{equation*}
\label{eq:gm_conds}
g(\msW^{\text{ur}}_m) \subseteq \msW^{\text{ur}, \perp}_m,
\end{equation*}
since $H^1_{\text{ur}}(G_v, \tfrac{1}{m}\Z/\Z)$ is orthogonal to itself under local Tate duality (this can also be checked by noting that, if the extension $F_v(\alpha^{1/m})/F_v$ is unramified for some choice of prime $v$ and $\alpha \in F_v^{\times}$, then $\alpha$ must have valuation divisible by $m$). In particular, we obtain a morphism
\begin{equation} \label{eq:g_morphism_SMod}
 g_m:\underline{\tfrac{1}{m}\Z/\Z}  \longrightarrow \underline{\mu_m},
\end{equation}
 which is readily checked to be self-dual. Using Proposition \ref{prop:agrees_with_artin_map}, we can  restate Proposition \ref{prop:cg_sym_intro} in the following form.
\begin{prop}
\label{prop:cg_Sel_sym}
Choose positive integers $a, b$ such that $F$ contains $\mu_{ab}$, and take $E(a, b)$ to be the exact sequence defined in \eqref{eq:Eab_def}. Then we have
\[\CTP_{E(a, b)}(\phi, g_a(\psi)) = \CTP_{E(b, a)}(\psi, g_b(\phi))\quad\text{for all }\, \phi \in \Cl^*F[b], \,\, \psi \in \Cl^*F[a].\]
In particular, in the case that $a = b$, the pairing $\CTP_{E(a, a)}(\text{---}, g_{a}(\text{---}))$ is symmetric.
\end{prop}
\begin{proof}
From the above discussion  we have  a morphism of exact sequences
\[
\begin{tikzcd}[column sep =small]
 E(b, a)\,=\,\big[\,0 \arrow{r} &\underline{\tfrac{1}{b}\Z/\Z} \arrow{d}{g_b } \arrow{r}{}& \underline{\tfrac{1}{ab}\Z/\Z}\arrow{d}{ g_{ab} }    \arrow{r}& \underline{\tfrac{1}{a}\Z/\Z}\arrow{d}{g_a }    \arrow{r} & 0\,\big] \\
 E(a, b)^{\vee}\,=\,\big[\,0 \arrow{r} & \underline{\mu_b}  \arrow{r} & \underline{\mu_{ab}}  \arrow{r} & \underline{\mu_a} \arrow{r} & 0\,\big].
\end{tikzcd}\]
Given $\phi$ and $\psi$ as in the proposition statement, we can combine the duality identity and naturality to give 
\[\CTP_{E(a, b)}(\phi, g_a(\psi)) = \CTP_{E(b, a)^{\vee}}(g_a(\psi), \phi) =  \CTP_{E(b, a)}(\psi, g_b^{\vee}(\phi)).\]
Since $g_b$ is self-dual, we have the result. 
\end{proof}
\begin{proof}[Proof of Proposition \ref{prop:cg_sym_intro}]
Take $(\alpha) = I^{a}$ and $(\beta) = J^{b}$ as in Proposition \ref{prop:cg_sym_intro}. Noting that $J$ and $\pi_{\text{Cl},b}(\beta)$ represent the same class in $\Cl\, F$, and noting that $\chi_{b, \beta}=\delta_b(\beta)$ by definition, we have
\begin{align*}
 g^{-1}\circ \chi_{a, \alpha}(\text{rec}_F(J)) \,=&\, \text{RP}(g^{-1} \circ \chi_{a, \alpha}, \,J) \\
=&\, -\CTP_{E(b, a)}(g^{-1} \circ \chi_{a, \alpha},\, \chi_{b, \beta}) \quad\text{from Proposition \ref{prop:agrees_with_artin_map}.}
\end{align*}
Repeating this argument for $\chi_{b, \beta}(\text{rec}_F(I))$, the claimed identity
\[ \chi_{a, \,\alpha}\left(\textup{rec}_F\left(J\right)\right) = \chi_{b, \,\beta}\left(\textup{rec}_F\left(I\right)\right)\]
is equivalent to
\[\CTP_{E(b, a)}(g^{-1} \circ \chi_{a, \alpha},\, \chi_{b, \beta}) = \CTP_{E(a, b)}(g^{-1} \circ \chi_{b, \beta},\, \chi_{a, \alpha}),\]
which is a consequence of Proposition \ref{prop:cg_Sel_sym}.  
\end{proof}

\subsection{Proof of Proposition \ref{prop:cg_alt_intro}} \label{sec:theta_complement_class}
We now turn to proving Proposition \ref{prop:cg_alt_intro}. Fix $F$, $\kappa$, $m$, $d$, $\alpha$, and $I$ as in that proposition. Given Proposition \ref{prop:agrees_with_artin_map}, we wish to prove the identity
\begin{equation}
\label{eq:cg_alt_v3}
d \cdot \CTP_{E(m, m)}( g^{-1} \circ \chi_{m, \,\alpha},\,\kappa( \chi_{m, \,\alpha})) = 0,
\end{equation}
where $g$ is the fixed  isomorphism \eqref{eq:root_of_unity_map_g}.

\begin{lem}
To prove Proposition \ref{prop:cg_alt_intro}, it suffices to establish \eqref{eq:cg_alt_v3} in the case where $d = 1$ and $F$ contains $\mu_{m^2}$.
\end{lem}
\begin{proof} 
Take $\zeta$ to be a fixed generator of $\mu_{2m^2/d}$. By splitting into cases depending on whether $\zeta$ lies in $F$, we can check that the map
\[a\zeta + b \,\mapsto\, \kappa(a) \zeta^{-1} + \kappa(b) \quad\text{for }\, a, b \in F\]
defines a field automorphism of $F(\mu_{2m^2/d})$. Repeating as necessary, we find that $\kappa$ extends to a field involution on $F(\mu_{m^2})$ that acts as $x \mapsto x^{-1}$ on $\mu_{m^2}$.

Using a subscript to record the field we are working over, we use the identity $\chi_{K,\, m, \,\alpha} = \res_{G_K} \chi_{F, \,m,\, \alpha}$ and Corollary \ref{cor:cor} to give
\begin{align*}
\left[F(\mu_{m^2})\,\colon \,F\right] \, &\cdot\,\CTP_{F, \,E(m, m)}(g^{-1} \circ  \chi_{F, \,m,\, \alpha}, \,  \kappa ( \chi_{F, \,m,\, \alpha}))\\
&=\, \CTP_{K,\, E(m, m)}(g^{-1} \circ  \chi_{K, \,m,\, \alpha}, \,  \kappa ( \chi_{K, \,m,\, \alpha})).
\end{align*}
We thus can establish the proposition for $F$ by applying \eqref{eq:cg_alt_v3} for $K$ with $d = 1$. This proves the lemma.
\end{proof}

Without loss of generality, we now restrict to the case $d = 1$.
\begin{notat}
For every positive integer $n$, take $\msW^0_n$ to be the subset of $\msW_n^{\text{ur}}$ consisting of tuples $(\phi_v)_v$ that are trivial at all places dividing $2$.

Take $E$ to be the exact sequence
 \[E\, =\, \left[0 \to \left(\tfrac{1}{m}\Z/\Z, \,\msW_m^{0}\right) \xrightarrow{\quad} \left(\tfrac{1}{m^2}\Z/\Z,\, \msW_{m^2}^{0}\right)  \xrightarrow{\,\,\cdot m\,\,} \left(\tfrac{1}{m}\Z/\Z, \,\msW_m^{0}\right)  \to 0\right]\]
in $\SMod_F$.

Take $F_0$ to be the subfield of $F$ fixed by the involution $\kappa$. Then $F/F_0$ is a quadratic extension. Recall that we defined a functor $\Ind_{F/F_0}$ from $\SMod_F$ to $\SMod_{F_0}$ in Definition \ref{defn:Ind}. Writing this functor simply as $\Ind$, we have an exact sequence
\[\Ind\, E \,=\, \left[0 \to \Ind\,\tfrac{1}{m}\Z/\Z \xrightarrow{\quad}  \Ind\, \tfrac{1}{m^2}\Z/\Z \xrightarrow{\,\,\cdot m\,\,}  \Ind\, \tfrac{1}{m}\Z/\Z \to 0\right]\]
in $\SMod_{F_0}$, where we have omitted the local conditions decorating these objects.
\end{notat}

Our next construction uses the language of theta groups detailed in \cite[Section 5]{MS21}.
\begin{notat}
\label{notat:class_theta}
Choose any positive divisor $n$ of $m^2$, and take $\zeta = g(1/n)$. Any element in $\Ind\, \tfrac{1}{n}\Z/\Z$ can be written as a linear combination of the elements
\[[1] \otimes \tfrac{1}{n} \quad\text{and}\quad [\kappa] \otimes \tfrac{1}{n} = \kappa\left([1] \otimes \tfrac{1}{n}\right).\]
We will denote these elements by $x_0$ and $y_0 = \kappa x_0$, respectively. In the case where we need to distinguish between different $n$, we will write them as $x_{0,n}$ and $y_{0,n}$.

Define $\mcalH_-(n)$ to be the group
\[\big \langle x, y, z \,\big|\,\, [x, z] = [y, z] =  x^n= y^n =  z^n = 1, \,\,[x,y]= z\big\rangle.\]
This is a quotient of the discrete Heisenberg group. It fits in the exact sequence
\begin{equation}
\label{eq:theta_cg_Ind_es}
1 \to \mu_n \to \mcalH_-(n) \to \Ind\,  \tfrac{1}{n}\Z/\Z \to 0
\end{equation}
where the fixed primitive root of unity $\zeta$ in $\mu_n$ is sent to $z$, where $x$ is sent to $x_0$, and where $y$ is sent to $y_0$. We can check that the $\Gal(F/F_0)$ action of $\mcalH_-(n)$ given by
\[\kappa x = y, \quad \kappa y = x, \quad \kappa z = z^{-1}\]
is well-defined, and that it makes the two maps in \eqref{eq:theta_cg_Ind_es} equivariant. The pushout of \eqref{eq:theta_cg_Ind_es} along $\mu_n \hookrightarrow (F^s)^{\times}$  defines a theta group
\begin{equation} \label{eq:theta_for_class_sequence}
1 \to (\Fsep)^{\times} \to  \mcalH(n) \to  \Ind\,  \tfrac{1}{n}\Z/\Z  \to 0,
\end{equation}
and the associated map $f_{\mathcal{H}(n)}$ (see \cite[Definition 5.1]{MS21}) is the isomorphism  
\[f_{\mcalH(n)}\colon \Ind\,\tfrac{1}{n}\Z/\Z \isoarrow\left (\Ind\,\tfrac{1}{n}\Z/\Z\right)^{\vee}\]
defined by
\begin{align*}
&f_{\mcalH(n)}(x_0)(x_0) = 1,\quad f_{\mcalH(n)}(y_0)(y_0) = 1,\\
& f_{\mcalH(n)}(y_0)(x_0) = \zeta,\quad f_{\mcalH(n)}(x_0)(y_0) =\zeta^{-1}.
\end{align*}
One sees from this that $f_{\mathcal{H}(n)}$  is equal to the composition of the maps
\begin{equation}
\label{eq:theta_pk}
\Ind\,\tfrac{1}{n}\Z/\Z  \xrightarrow{\,\,\rho_{\kappa}\,\,}\Ind\,\tfrac{1}{n}\Z/\Z  \xrightarrow{\,\,\Ind\, g\,\,} \Ind\, \mu_n  \xrightarrow{\,\,t\,\,}  \left(\Ind\,\tfrac{1}{n}\Z/\Z\right)^{\vee},
\end{equation}
where $t$ is defined as in \eqref{eq:where_t_is_defined} and $\rho_{\kappa}$ is defined  as in Figure \eqref{fig:nat_mor}.
\end{notat}

For the statements of the following  results, we refer to \cite[Definition 5.3]{MS21} for the notion of isotropic  local conditions, and to \cite[Definition 5.7]{MS21} for the definition of the Poonen--Stoll class.

\begin{lem}
\label{lem:class_isotropy} 
Given the setup of Notation \ref{notat:class_theta}, the local conditions $\shap_{\textup{loc}}(\msW^0_n)$ are isotropic with respect to the theta group $\mathcal{H}(n)$.
\end{lem}
\begin{proof}
Choose $(\phi_v)_v$ in  $\shap_{\textup{loc}} (\msW^0_n)$. Given any place $v$ of $F_0$, we wish to show that $q_{G_v}(\phi_v)$ is zero, where 
\[q_{G_v}: H^1\left(G_v,\Ind \tfrac{1}{n}\mathbb{Z}/\mathbb{Z}\right)\longrightarrow H^2\left(G_v,(F_v^s)^\times\right)\] 
is the connecting map on cohomology associated to \eqref{eq:theta_for_class_sequence} (cf. \cite[Notation 5.2]{MS21}). 

If $F/F_0$ is not ramified at $v$, then $q_{G_v}(\phi_v)$ lies in the image of $H^2_{\text{ur}}(G_v, \mu_n)$, and this group is trivial. So we can assume $F/F_0$ is ramified at $v$.

Now $F_0(\mu_4)$ is a quadratic extension of $F_0$ contained in $F$, so it equals $F$. Consequently, $F/F_0$ is only ramified at infinite places and places dividing $2$. By the definition of $\msW^0_n$, $\phi_v$ is trivial at all such places, so $q_{G_v}(\phi_v)$ is zero. This gives the lemma.
\end{proof}

\begin{lem}
\label{lem:theta_nogoody}
The Poonen--Stoll class corresponding to the exact sequence $\Ind\, E$ and the theta group $\mcH(m^2)$ is trivial.
\end{lem}
\begin{proof}
Taking $\mcH_1$ to be the preimage of $\Ind\, \tfrac{1}{m}\Z/\Z$ in $\mcH(m^2)$, we define a homomorphism
\[s \colon \Ind\, \tfrac{1}{m}\Z/\Z \to \mcH_1\]
by mapping $x_{0, m}$ to $x^m$ and $y_{0, m}$ to $y^m$, where $x$ and $y$ are taken from the generators of $\mcH(m^2)$. The map $s$ is equivariant, and is a section of the natural map from $\mcH_1$ to $\tfrac{1}{m}\Z/\Z$. The cocycle $\overline{\psi}_{\text{PS}}$ defined in \cite[Remark 5.8]{MS21} from $s$ is thus trivial, from which the result follows.
\end{proof}

\begin{proof}[Completion of the proof of Proposition \ref{prop:cg_alt_intro}]
The map $f_{\mcH(m^2)}$ above fits in the diagram
\begin{equation}
\label{eq:fhm2_diagram}
\begin{tikzcd}[column sep =small]
 0 \arrow{r} & \Ind\, \tfrac{1}{m}\Z/\Z \arrow{d}{f_{\mcH(m)}}\arrow{r} & \Ind\, \tfrac{1}{m^2}\Z/\Z\arrow{d}{f_{\mcH(m^2)}} \arrow{r}& \Ind\, \tfrac{1}{m}\Z/\Z \arrow{d}{f_{\mcH(m)}} \arrow{r} & 0  \\
 0 \arrow{r} & \left(\Ind\, \tfrac{1}{m}\Z/\Z\right)^{\vee}  \arrow{r} & \left(\Ind\, \tfrac{1}{m^2}\Z/\Z\right)^{\vee}  \arrow{r} & \left(\Ind\, \tfrac{1}{m}\Z/\Z\right)^{\vee} \arrow{r} & 0
\end{tikzcd}
\end{equation}
in $\SMod_{F_0}$. Since the local conditions $\shap_{\text{loc}}(\msW_{m^2}^0)$ are isotropic for $\mcH(m^2)$, and since the associated Poonen--Stoll class is trivial, \cite[ Theorem 5.10]{MS21} gives
\begin{equation}
\label{eq:cg_alt_v4}
\CTP_{\Ind\, E}(\phi, \, f_{\mcH(m)}(\phi)) = 0\quad\text{for all }\,\phi \in \Sel_{F_0} \left(\Ind\, \tfrac{1}{m}\Z/\Z,\, \shap_{\textup{loc}} (\msW^0_m)\right).
\end{equation}
Taking $\phi = \shap_{F/F_0}(\psi)$ for some $\psi$ in $\Sel\left( \tfrac{1}{m}\Z/\Z,\, \msW^0_m\right)$, we have
\begin{alignat*}{2}
0 & = \,\CTP_{\Ind\, E}(\phi, \, f_{\mcH(m)}(\phi)) &&  \\
&= \,\CTP_{\Ind\, E}(\phi, \, t\circ \Ind\, g\circ  \rho_{\kappa}(\phi)) && \text{ from  \eqref{eq:theta_pk},}\\
 &=\, \CTP_E(\shap(\psi),\, t \circ \Ind\, g \circ\shap( \kappa \psi))\,\, && \text{ by commutativity of \eqref{eq:rho_tau_square},} \\
& = \,\CTP_E(\shap(\psi),\, t \circ \shap\circ g(\kappa \psi)) && \text{ by naturality of the Shapiro isomorphism,}\\
&=\, \CTP_E(\psi, \, g(\kappa \psi) )&& \text{ by Theorem \ref{thm:fields_main}.}
\end{alignat*}
This establishes the proposition.
\end{proof}

\section{Bilinearly enhanced class groups}  
\label{ssec:LST}
In \cite{Lipn20}, Lipnowski, Sawin, and Tsimerman generalized the Cohen--Lenstra heuristics to certain families of fields containing many roots of unity. As part of this work, they introduced the notion of a bilinearly enhanced group, which we give now. The equivalence between this definition and the one appearing in \cite[Definition 2.1]{Lipn20} is established by \cite[Lemma 4.2]{Lipn20}.
\begin{notat}
Take $\ell^n$ to be a positive power of a rational prime and choose a finite abelian $\ell$-group $G$. Writing $G^*$ for $\Hom(G, \QQ_{\ell}/\Z_{\ell})$, choose a homomorphism
\[\psi \colon G^*[\ell^n] \to G[\ell^n]\]
and  an infinite sequence $\omega_1, \omega_2, \dots$ of alternating pairings
\[\omega_i\colon G^*[\ell^i] \otimes G^*[\ell^i] \to \QQ_{\ell}/\Z_{\ell}.\]
Taking 
\[\langle\text{--},\,\text{--}\rangle \colon G \times G^* \to \QQ_{\ell}/\Z_{\ell}\]
to be the evaluation pairing, we call the tuple $(G, \psi, (\omega_i)_{i\ge 1})$ an \emph{$\ell^n$-bilinearly enhanced group} if
\begin{equation}
\label{eq:LST_beg_1}
\langle \psi(\ell^r\alpha), \, \beta \rangle - \langle \psi(\ell^r\beta), \, \alpha \rangle  = 2 \cdot \omega_{n+r}(\alpha, \beta)
\end{equation}
for all $r \ge 0$ and $\alpha, \beta \in G^*[\ell^{n+r}]$, and if
\begin{equation}
\label{eq:LST_beg_2}
\omega_i(\ell \alpha, \beta) = \omega_{i+1}(\alpha, \beta)
\end{equation}
for all $i \ge 1$, $\alpha \in G^*[\ell^{i+1}]$, and $\beta  \in G^*[\ell^i]$.
\end{notat}
Given a number field $F$ containing $\mu_{\ell^n}$ with $\ell$ an odd prime, Lipnowsk et al. give an associated $\ell^n$-bilinearly enhanced group 
\begin{equation}
\label{eq:LST_cg_beg}
\left(\Cl\,F[\ell^{\infty}], \, \psi_F, \,(\omega_{F, i})_i\right).
\end{equation}
Here, the map $\psi_F\colon \Cl^*F[\ell^n] \to \Cl\,F[\ell^n]$ is straightforward to define, though it does require a choice of a primitive $(\ell^n)^{th}$ root of unity. In our notation, it is the composition of the maps
\begin{align*}
\Cl^*F[\ell^n] \,=\,\, &\Sel~ \underline{\tfrac{1}{\ell^n}\Z/\Z}  \, \xrightarrow{\,\,g_{\ell^n}\,\,} \,\Sel~\underline{\mu_{\ell^n}}\, \isoarrow \,\Sel^{\ell^n} F\, \xrightarrow{\,\,\pi_{\text{Cl}}\,\,}\, \Cl\, F[\ell^n].
\end{align*} 
The condition \eqref{eq:LST_beg_1} determines the pairings $\omega_i$ for $i \ge n$. For $i < n$, the definition of $\omega_i$ is somewhat subtler. In \cite{Lipn20}, it is defined using Artin--Verdier duality between the \'{e}tale cohomology groups
\[H^1_{\text{\'{e}t}}\left(\text{Spec}\, \mcO_F, \,\Z/\ell^i\Z\right)\quad\text{and}\quad H^2_{\text{\'{e}t}}\left(\text{Spec}\, \mcO_F, \,\mu_{\ell^i}\right).\]
Having defined these objects, Lipnowski et al. then give a heuristic for how often a given $\ell^n$-bilinearly enhanced group appears as the tuple \eqref{eq:LST_cg_beg} as $F$ varies through certain families of number fields containing $\mu_{\ell^n}$.

The existence of the bilinearly enhanced group \eqref{eq:LST_cg_beg} recovers Proposition \ref{prop:cg_sym_intro}  in the case that $\ell$ is odd. For suppose we have chosen nonnegative integers $c$ and $d$ satisfying $c+ d =n$. Then, for $\alpha$ in $\Cl^* F[\ell^c]$ and $\beta$ in $\Cl^*F[\ell^d]$, we have
\begin{equation}
\label{eq:LST_to_MS}
\left\langle \psi_F(\alpha), \,\beta \right\rangle\, -\,\left \langle \psi_F(\beta), \,\alpha\right\rangle \,=\,\omega_{F, n}(\alpha, \beta) \,=\, \omega_{F, n-c}(\ell^c \alpha, \beta) \,=\, 0.
\end{equation}

This is the full extent of what we can conclude about the class group of $F$ and its reciprocity pairing from the mere existence of the bilinearly enhanced group \eqref{eq:LST_cg_beg}.  Specifically, we have the following proposition, whose proof we omit.
\begin{prop}
Take $\ell^n$ to be a positive power of an odd prime. Choose a finite abelian $\ell$-group $G$ and a homomorphism
\[\psi \colon G^*[\ell^n] \to G[\ell^n].\]
Then the following are equivalent:
\begin{enumerate}
\item There is some sequence of alternating pairings $\omega_1, \omega_2, \dots$ so that $(G, \psi, (\omega_i)_{i \ge 1})$ is an $\ell^n$-bilinearly enhanced group.
\item For every pair of positive integers $c, d$ with sum $n$, and for every $\alpha$ in $G^*[\ell^c]$ and $\beta$ in $G^*[\ell^d]$, we have
\[\langle \psi(\alpha),\beta \rangle = \langle \psi(\beta), \alpha\rangle.\]
\end{enumerate}
\end{prop}
 
There is generally more than one way to complete $(\Cl\, F[\ell^{\infty}], \psi_F)$ to a bilinearly enhanced group. If we want to choose one canonically, we may do so using the Cassels--Tate pairing. This gives an alternative to the approach considered in \cite[4.2.2]{Lipn20}. It seems likely that our pairings agree with those constructed in \cite{Lipn20}.

\begin{notat}
\label{notat:our_beg}
Choose $n$ a positive power  of $\ell$  so $F$ contains $\mu_n$, and fix $g$ as in \eqref{eq:root_of_unity_map_g}.  

Take $E_n$ to be the pushout of $E(n, n)$ along the map $g_n \colon \underline{\tfrac{1}{n}\Z/\Z} \to\underline{\mu_n}$. This exact sequence takes the form
\[E_n\,=\,\left[0 \to \underline{\mu_n} \xrightarrow{\,\,g^{-1}\,\,}  \left(\tfrac{1}{n^2}\Z/\Z, \,\msW_n^{\text{ur}, \perp} + \msW_{n^2}^{\text{ur}}\right) \to \underline{\tfrac{1}{n}\Z/\Z} \to 0\right].\]
The dual sequence takes the form
\[E^{\vee}_n\,=\,\left[0 \to\underline{\mu_n}  \to \left(\mu_{n^2}, \,(\msW_n^{\text{ur}, \perp} + \msW_{n^2}^{\text{ur}})^{\perp}\right) \to \underline{\tfrac{1}{n}\Z/\Z}\to 0\right].\]
These are both extensions of $\underline{\tfrac{1}{n}\Z/\Z} $ by $\underline{\mu_n}$. Furthermore, $E_n$ and $E_n^{\vee}$ are isomorphic as sequences of abelian groups.

We then  define the object $M_n$ in $\SMod_F$ as the central term of the Baer difference of these sequences, so we have an exact sequence
\[E_n - E^{\vee}_n = \left[0 \longrightarrow \underline{\mu_n} \longrightarrow M_n \longrightarrow \underline{\tfrac{1}{n}\Z/\Z} \longrightarrow 0\right].\]
Considered in the category of abelian groups, this exact sequence splits.  Furthermore, since Baer sum commutes naturally with taking duals, we can find a commutative diagram
\[
\begin{tikzcd}[column sep =small]
 E_n - E^{\vee}_n \,=\,\big[\,0 \arrow{r} &\underline{\mu_n} \arrow{d}{\text{Id}} \arrow{r}{}& M_n\arrow{d}{\gamma_n}    \arrow{r}& \underline{\tfrac{1}{n}\Z/\Z}\arrow{d}{-\text{Id}}    \arrow{r} & 0\,\big] \\
( E_n - E^{\vee}_n)^{\vee} \,=\,\big[\,0 \arrow{r} & \underline{\mu_n} \arrow{r} &M_n^\vee \arrow{r} & \underline{\tfrac{1}{n}\Z/\Z}\arrow{r} & 0\,\big] 
\end{tikzcd}\]
in $\SMod_F$, with the central column $\gamma_n$ an isomorphism in $\SMod_F$. We detail an explicit choice for $\gamma_n$ in  Remark \ref{rmk:explicit_M_m_extension}.

Take $m$ to be a positive divisor of $n$. Because $ E_n - E^{\vee}_n$ splits in the category of abelian groups, we can apply the functor $\Hom(\Z/m\Z, \text{---})$ to  produce a new exact sequence of $G_F$-modules, which we denote $(E^{\vee}_n - E_n)[m]$. We have a commutative diagram
\[
\begin{tikzcd}[column sep =small]
 ( E_n - E^{\vee}_n)[m] \,&=\,\big[\,0 \arrow{r} &\underline{\mu_m}   \arrow[hookrightarrow]{d}   \arrow{r}{} & (M_n[m], \msW_{n, m}) \arrow[hookrightarrow]{d} \arrow{r}& \underline{\tfrac{1}{m}\Z/\Z} \arrow[hookrightarrow]{d}    \arrow{r} & 0\,\big] \\
 E_n - E^{\vee}_n\,&=\,\big[\,0 \arrow{r} &\underline{\mu_n} \arrow{d}{n/m} \arrow{r}{}& M_n \arrow{d}{ n/m}    \arrow{r}& \underline{\tfrac{1}{n}\Z/\Z}\arrow{d}{n/m}   \arrow{r} & 0\,\big] \\
 ( E_n - E^{\vee}_n)[m] \,&=\,\big[\,0 \arrow{r} &\underline{\mu_m}\arrow{r}{} & (M_n[m], \msW^-_{n, m}) \arrow{r}& \underline{\tfrac{1}{m}\Z/\Z} \arrow{r} & 0\,\big].
\end{tikzcd}
\]
Here, we have assigned the local conditions $\msW_{n, m}$ in the top row so that the top three vertical maps are all strictly monic in $\SMod_F$. The local conditions $\msW_{n, m}^-$ in the bottom row are assigned so the bottom three vertical maps are all strictly epic.

We immediately see that $\msW_{n, m}^-$ is contained in $\msW_{n, m}$. By easy nonsense, the final map in the bottom row is strictly epic. We then have that the final map in the top row is strictly epic. Similarly, the first map in the first row is strictly monic, so the first map in the final row is also strictly monic. So all rows are exact, and $\msW_{n,m}$ equals $\msW_{n, m}^-$.

The morphism $(M_n[m], \msW_{n, m}) \hookrightarrow M_n$, and the composition of morphisms
\[ (M_n[m], \msW_{n, m})^{\vee} \xrightarrow{\,\, (n/m)^{\vee}\,\,} M_n^\vee \xrightarrow{\,\,\gamma_n^{-1}\,\,}M_n ,\]
are both kernels of the multiplication by $m$ morphism. So there is a unique isomorphism \[\gamma_{n, m}: (M_n[m], \msW_{n, m}) \longrightarrow (M_n[m], \msW_{n, m})^{\vee}\]
in $\SMod_F$ that commutes with the two kernel morphisms. This map fits in the commutative diagram (in $\SMod_F$) with exact rows

\begin{equation} \label{eq:our_beg_anti}
  \begin{tikzcd}[column sep =small]
 (E_n - E^{\vee}_n)[m] \,=\,\big[\,0 \arrow{r} &\underline{\mu_n} \arrow{d}{\text{Id}} \arrow{r}{}& (M_n[m], \msW_{n, m}) \arrow{d}{\gamma_{n,m}}    \arrow{r}& \underline{\tfrac{1}{m}\Z/\Z}\arrow{d}{-1}    \arrow{r} & 0\,\big] \\
( E_n - E^{\vee}_n)[m]^{\vee} \,=\,\big[\,0 \arrow{r} & \underline{\mu_m} \arrow{r} &(M_n[m], \msW_{n,m})^{\vee}  \arrow{r} & \underline{\tfrac{1}{m}\Z/\Z}\arrow{r} & 0\,\big].
\end{tikzcd}  
\end{equation}  
\end{notat}

\begin{prop}
\label{prop:our_beg}
Choose a global field $F$ and a prime $\ell$ not equal to the characteristic of $F$.
Take $n$ to be a power of $\ell$ so $F$ contains $\mu_{\ell^{n}}$. Given a positive divisor $m$ of $n$, take $D(m)$ to be the exact sequence $(E_n - E_n^{\vee})[m]$.

Then, for all divisors $m$ of $n$, the pairing
\[\CTP_{D(m)} \colon \Cl^*F[m] \otimes \Cl^*F[m] \to \QQ/\Z\]
is antisymmetric.

Further, if $a$ and $b$ are positive integers satisfying $m = ab$, we have
\begin{equation}
\label{eq:our_beg_2}
\CTP_{D(m)}(\phi, \psi) = \CTP_{D(a)}(b\phi, \psi)
\end{equation}
for all $\phi$ in  $\Cl^*F[m]$ and $\psi$ in $\Cl^*F[a]$.

Finally, we have
\[\CTP_{D(n)}(\phi, \psi) = \CTP_{E(n, n)}(\phi, g(\psi)) - \CTP_{E(n, n)}(\psi, g(\phi))\]
for all $\phi$ and $\psi$ in  $\Cl^*F[n]$.
\end{prop}
\begin{proof}
The antisymmetry of $\CTP_{D(m)}$ follows from the duality identity and from applying naturality to the morphism $D(m) \to D(m)^{\vee}$ given in \eqref{eq:our_beg_anti}.

The relation \eqref{eq:our_beg_2} is the statement of naturality for the morphism $D(a) \to D(m)$ of exact sequences defined by  the natural inclusion at each entry.

The final claim follows from the relation $E_n - E_n^{\vee} = D(n)$ and the duality identity applied to the sequence $E_n^{\vee}$.
\end{proof}

If $\ell$ is odd, this proposition shows that the pairings $\CTP_{D(m)}$ may be used to complete the tuple $(\Cl^*F[\ell^{\infty}], \psi_F)$ to a bilinearly enhanced group.

\begin{rmk} \label{rmk:explicit_M_m_extension}
From the definition of the Baer sum, the underlying $G_F$-module of the central term $M_n$ of $E_n - E_n^{\vee}$ is the subquotient
\[M_n=\frac{\{(x,y)\in  \mu_{n^2}\oplus \frac{1}{n^2}\mathbb{Z}/\mathbb{Z}~~\colon~~ny=g^{-1}(x^n)\}}{\{(x, g^{-1}(x))~~\colon~~x\in \mu_n\}}\]
of  $\mu_{n^2}\oplus \frac{1}{n^2}\mathbb{Z}/\mathbb{Z}$.
The injection $\mu_n \to M_n$ sends $x$ to the class of $(x, 0)$, and the surjection $M_n \to \tfrac{1}{n}\Z/\Z$ sends a class represented by $(x, y)$ to $ny$.

At the level of $G_F$-modules, one can take $\gamma_n \colon M_n \to (M_n)^{\vee}$ to be induced by the map
\[\gamma_n\colon  \mu_{n^2}\oplus \tfrac{1}{n^2}\mathbb{Z}/\mathbb{Z} \longrightarrow   \tfrac{1}{n^2}\mathbb{Z}/\mathbb{Z} \oplus \mu_{n^2}\]
sending $(x,y)$ to $(-y, x )$. This map corresponds to an alternating pairing on $M_n$.
\end{rmk}

In the case $\ell = 2$, we can complement Proposition \ref{prop:our_beg} as follows. 

\begin{prop}
\label{prop:Dm_alt}
Suppose $F$ contains $\mu_n$ for some $n$ a power of  $2$. Then, for any positive divisor $m$ of $n$, the pairing $\textup{CTP}_{D(m)}$ is alternating.  
\end{prop}

\begin{proof}
For $m=n$ this follows from the final part of Proposition \ref{prop:our_beg}. Thus we suppose $m<n$. Since the case $m = 1$ is trivial, we may also assume $m > 1$.

Take $P_1: M_n[2m] \times M_n[2m] \to \mu_{2m}$ to be the alternating pairing corresponding to the isomorphism $\gamma_{n, 2m}$. Following \cite[Definition 5.17]{MS21} with this pairing and $e =1$ gives a theta group $\mcH$ over $M_n[m]$ constructed as a quotient $U/N$. In the terminology of \cite[Definition 5.1]{MS21}, the map $f_\mcH$ associated to this theta group equals $\gamma_{n, m}$. The composition of the morphisms
\[M_n\xrightarrow{\,\, n/2m\,\,} \left(M_n[2m], \,\msW_{n, 2m}\right)  \xrightarrow{\,\, \,2\,\,\,}  \left(M_n[m], \,\msW_{n, m}\right)\]
is strictly epic, so the second morphism in this sequence is also strictly epic. By \cite[Example 5.18]{MS21}, we see that $\msW_{n, m}$ is an isotropic subgroup of local conditions for the theta group $\mcH$.

The group $U$ is isomorphic as a $G_F$ set to $(F^s)^{\times} \times M_n[2m]$, and we can check that the map $\zeta \mapsto (1, \zeta)$ defines a $G_F$-equivariant homomorphism from $\mu_{2m}$ to $U$ fitting into the triangle
\[\begin{tikzcd}
& U \arrow{d} \\
\mu_{2m} \arrow{r} \arrow{ur} & M_n[2m].
\end{tikzcd}\]
This map descends to a $G_F$-equivariant lift $\mu_m \to \mcH$ of the map $\mu_m \to M_n[m]$, so we find the Poonen--Stoll class  associated to $D(m)$ and $\mcH$ by \cite[Definition 5.7]{MS21} is $0$. The proposition now follows from \cite[Theorem 5.10]{MS21}. \end{proof}

\section{Decomposing class groups}
\label{ssec:class_decomp}

The Cohen--Lenstra--Martinet heuristics \cite{CoLe84, CoMa90} give a conjectural description for how the class group of an extension of global fields $K/F$ varies as $K$ changes within some family. According to these conjectures, the structure of the semisimple decomposition of the $G_F$-representation $\QQ[G_F/G_K]$ has an impact on how the class group of $K$ behaves. Using Construction \ref{ctn:class_twist} below, the Cassels--Tate pairing seems to be a helpful tool for breaking apart the class group with respect to such a semisimple decomposition. 

In what follows, we will work in the categories $\SMod_{F, \ell}^{\infty}$ and $\SMod_{K, \ell}^{\infty}$ introduced in \cite[Section 4.2]{MS21}, so as to allow the underlying modules to be infinite. As in Remark \ref{rmk:infinite_F_change}, the results on base change proved in Section \ref{sec:fields} remain valid in this setting.

\begin{ctn}
\label{ctn:class_twist}
Choose a separable extension $K/F$ of global fields, and take $\ell$ to be a rational prime other than the characteristic of $F$. Taking $A = \QQ_{\ell}/\Z_{\ell}$, we have an identification
\[\Cl^* K\left[\ell^{\infty}\right] \,\cong\, \Sel_K\left(A, \msW^{\text{ur},\, K}\right) \,\cong\, \Sel_F\, \left(\Ind_{K/F}A,\, \msW^{\text{ur},\, F}\right),\]
where the first isomorphism is as defined Section \ref{ssec:fund_class_groups}, and the second isomorphism is the Shapiro isomorphism of Section \ref{sec:fields}.  Here, $\msW^{\text{ur},\, K}$ and $\msW^{\text{ur},\, F}$ are the subgroups of everywhere unramified local conditions for $A$ and $\Ind_{K/F}A$. They are identified with each other by the local Shapiro isomorphism as in \eqref{eq:KF_unr}.

Choose an isogeny of $G_F$-modules  
\begin{equation}
\label{eq:fake_wedder}
\Ind_{K/F}\, \Z_{\ell} \hookrightarrow \bigoplus_{i  \le r} T_i,
\end{equation}
where $T_1, \dots, T_r$ is a sequence of compact finitely generated free $\Z_{\ell}$-modules with a continuous action of $G_F$. Here by the term isogeny we mean that the map is injective and has finite cokernel.  Equivalently, taking $V_i =\QQ_{\ell} \otimes T_i  $ and $A_i = A \otimes T_i$, the corresponding map
\[\Ind_{K/F} \,\QQ_{\ell} \xrightarrow{\quad\,\,} \bigoplus_{i \le r} V_i\]
is an isomorphism. 

Taking $M$ to be the cokernel of \eqref{eq:fake_wedder}, the snake lemma gives an exact sequence of $G_F$-modules
\[0 \to M \xrightarrow{\,\,\iota\,\,} \Ind_{K/F} A \xrightarrow{\,\,\pi\,\,} \bigoplus_{i \le r} A_i \to  0.\]
We give $\Ind_{K/F} A$ the unramified local conditions and give $M$ and $\bigoplus_{i \le r} A_i$ the local conditions so that this sequence is exact.

From \cite[Proposition 4.11]{MS21}, we have an exact sequence
\begin{equation}
\label{eq:class_decomp_main}
\Sel_F\, M \to \Cl^* K\left[\ell^{\infty}\right] \to \bigoplus_{i \le r} \Sel_F\, A_i \to \Hom\left(\Sel_F M^{\vee},\,\,\QQ/\Z\right),
\end{equation}
with the last map given by the Cassels--Tate pairing.
\end{ctn}

\begin{rmk}
In the above, to calculate the local conditions on $\bigoplus_{i \le r} A_i$  more precisely, we start with the commutative square
\[\begin{tikzcd}
H^1_{\text{ur}}(G_v, \,\Ind_{K/F}\,\QQ_{\ell}) \ar{r}{\sim} \ar{d} & \bigoplus_{i \le r} H^1\left(G_v/I_v, \, V_i^{I_v}\right) \ar{d}{\text{inf}}\\
H^1_{\text{ur}}(G_v, \Ind_{K/F}A)  \ar{r} & \bigoplus_{i \le r} H^1(G_v, \, A_i),
\end{tikzcd}\]
which we have for every place $v$ of $F$. The first vertical map is surjective since $H^1_{\text{ur}}(G_w, \QQ_{\ell})$ surjects onto $H^1_{\text{ur}}(G_w, A)$ for each place $w$ of $K$. So the local conditions for $\bigoplus_{i \le r}  A_i$ are given by
\[\prod_{v\text{ of } F} \bigoplus_{i \le r} \text{inf}\left(H^1(G_v/I_v, V_i^{I_v})\right).\]
\end{rmk}

We will consider Construction \ref{ctn:class_twist} in detail in the case that $K/F$ is cyclic of degree $\ell$. We expect that it would give similarly--detailed information in the case that $K/F$ is any abelian extension.
\begin{ex}
In Construction \ref{ctn:class_twist}, take $K/F$ to be a cyclic extension of degree $\ell$, and   
\[\chi\colon \Gal(K/F) \isoarrow \mu_{\ell}\]
to be a nontrivial character associated to this extension. Take  $T =\Z_{\ell}$ and $T^{\chi} = \Z_{\ell}[\mu_{\ell}]$. We give $T$ the trivial $G_F$-action and define an action on $T^{\chi}$ by
\[\sigma x = \chi(\sigma)\cdot x\quad\text{for all }\,\, \sigma \in G_F.\]
Take $A = \QQ_{\ell}/\Z_{\ell}$ as above, and define $A^{\chi}$ to be $A \otimes T^{\chi}$. We have an exact sequence
\[0 \to \tfrac{1}{\ell}\Z/\Z \xrightarrow{\,\,\iota\,\,} \Ind_{K/F} A \xrightarrow{\,\,(\nu, \, p)\,\,} A \oplus A^{\chi}\to 0\]
of $G_F$-modules, where $\iota$ is given by
\[\iota(a) = \sum_{i =0}^{\ell - 1}\, [\sigma^i] \otimes  a\quad\text{for all }\,\,a \in A,\]
where $\nu$ is given as in Figure \ref{fig:nat_mor}, and where $p$ is given by
\[p([\sigma] \otimes a) = \chi(\sigma)\cdot a\quad\text{for }\,\, a \in A, \,\sigma \in G_F.\]
Giving $\Ind_{K/F}A$ the subgroup of unramified local conditions, the resulting local conditions on $A^{\chi}$ are given by
\[\msW^{\chi} = \prod_{\substack{v \text{ of } F \\ K/F \text{ unram. at } v}} H^1_{\text{ur}}(G_v, A^{\chi}),\]
and the local conditions for $A$ are the unramified ones. The map $\iota$ is the composition  
\[ \tfrac{1}{\ell}\Z/\Z \,\xrightarrow{\,\,i_1\,\,}\, \Ind_{K/F}\, \tfrac{1}{\ell}\Z/\Z \,\hookrightarrow \,\Ind_{K/F}\, \QQ_{\ell}/\Z_{\ell}, \]
with $i_1$ taken as in Figure \ref{fig:nat_mor}. From \eqref{eq:cor_from_sh}, the local conditions for $\tfrac{1}{\ell}\Z/\Z$ are seen to be
\[\msW = \prod_{v\text{ of } F} \ker\Big(H^1\left(G_v,\,\tfrac{1}{\ell}\Z/\Z\right)  \xrightarrow{\,\,\text{res}\,\,} H^1\left(G_K \cap I_v,\,\tfrac{1}{\ell}\Z/\Z\right)\Big).\]
The exact sequence \eqref{eq:class_decomp_main} then takes the form
\[
\Sel\left(\tfrac{1}{\ell}\Z/\Z ,\,\msW\right)\xrightarrow{\quad} \Cl^*K[\ell^{\infty}] \xrightarrow{\,\,(\cores,\,\, p \,\circ\, \shap)\,\,} \Cl^*F[\ell^{\infty}] \oplus \Sel \, A^{\chi} \xrightarrow{\quad}  \Sel\left(\mu_{\ell},\, \msW^{\perp}\right)^*,
\]
the $*$ denoting a Pontryagin dual. 
We note that, for $\sigma$ in $G_F$, $\sigma \colon A^{\chi} \to A^{\chi}$ is a morphism in $\SMod_{F, \ell}^{\infty}$, so $\Sel\, A^{\chi}$ has the structure of a $\Gal(K/F)$-module.

From this exact sequence, we can prove the following proposition.
\begin{prop}
\label{prop:class_decomp_ex}
In the above situation, we have an exact sequence 
\begin{equation}
\label{eq:class_decomp_ex_2}
0 \to \left(\Cl^*K[\ell^{\infty}]\right)^{\Gal(K/F)} \xrightarrow{\quad\,\,} \Cl^*K[\ell^{\infty}] \xrightarrow{\,\,\,p\,\,\,} \Sel \,A^{\chi} \to  \Sel\left(\mu_{\ell},\, \msW^{\perp} \right)^* 
\end{equation}
of $\Gal(K/F)$-modules,  where the leftmost map is the natural inclusion, and where $\Gal(K/F)$ acts trivially on the final term. 

 In particular, denote by $S$ the set of places of $F$ not dividing $l$ where $K/F$ is ramified, and  suppose that the following assumption is satisfied:
 \begin{itemize}
\item[$\left(\boldsymbol\star\right)$] There is no element of $F^{\times}$ outside $(F^{\times})^{\ell}$ that has valuation divisible by $\ell$ at all finite primes and which lies in $(F_v^{\times})^{\ell}$ at every prime $v$ in $S$.
\end{itemize}
Then we have an exact sequence of $\Gal(K/F)$-modules
\[0 \to \left(\Cl^*K[\ell^{\infty}]\right)^{\Gal(K/F)} \xrightarrow{\quad\,\,} \Cl^*K[\ell^{\infty}] \xrightarrow{\,\,\,p\,\,\,} \Sel \,A^{\chi} \to  0.\]
\end{prop}
\begin{proof}
To show \eqref{eq:class_decomp_ex_2} is exact, it suffices to show that the kernel of $p$ in $\Cl^*K[\ell^{\infty}]$ is the portion invariant under $G_F$.  So take $\sigma$ to be a generator of $\Gal(K/F)$, and take $\zeta = \chi(\sigma)$. We have an exact sequence
\[0 \to A^{\chi} \xrightarrow{\,\,f\,\,} \Ind_{K/F} A \xrightarrow{\,\,\nu\,\,} A \to 0\]
in which the homomorphism $f$ is given by 
\[f(a \otimes \zeta^i) \,=\,\left([\sigma^{i+1}] -[\sigma^{i}] \right) \otimes a\quad\text{for}\quad a \in A\,\text{ and }\, 0 \le i < \ell.\]
From the long exact sequence, we see that the map
\begin{equation}
\label{eq:last_injection}
H^1(G_F, A^{\chi}) \xrightarrow{\,\,f\,\,} H^1\left(G_F, \, \Ind_{K/F}A\right)
\end{equation}
is injective. In addition, the map $(\rho_{\sigma} - 1)\colon \Ind_{K/F}A  \to \Ind_{K/F}A $ factors as the composition
\[\Ind_{K/F}A \xrightarrow{\,\,p\,\,} A^{\chi}\xrightarrow{\,\,f\,\,} \Ind_{K/F}A.\]
From these two facts, we see that the kernel of $p$   is precisely the portion of $\Sel\left(\Ind_{K/F}A\right)$ invariant under $\rho_{\sigma}$. The exactness of \eqref{eq:class_decomp_ex_2} follows by \eqref{eq:rho_tau_square}, and the equivariance of this sequence under $\Gal(K/F)$ follows from the commutativity of the diagram
\[\begin{tikzcd}
0 \ar{r} & \tfrac{1}{\ell}\Z/\Z \ar{r}{\iota} \ar{d}{\text{Id}} & \Ind_{K/F} A \ar{r}{(\nu,\, p)} \ar{d}{\rho_{\sigma}} &  A \oplus A^{\chi}\ar{r} \ar{d}{\text{Id} \,\oplus\, \sigma} &0\\
0 \ar{r} & \tfrac{1}{\ell}\Z/\Z \ar{r}{\iota} & \Ind_{K/F} A \ar{r}{(\nu, \,p)}&  A \oplus A^{\chi} \ar{r}  &0.
\end{tikzcd}\]

Finally, if the assumption ($\boldsymbol \star$) is satisfied, we find that $\Sel\, (\mu_{\ell}, \msW^{\perp})$ vanishes by interpreting the local conditions using Hilbert 90. 
 This gives the second exact sequence, and the proposition follows.
\end{proof}
\end{ex}

\section{Two miscellaneous examples}
\label{sec:Two}

\subsection{Involution spin}
\label{ssec:Involution}
The notion of the spin of a prime ideal was introduced in \cite{FIMR13}. Given a number field $K$, a nontrivial field automorphism $\kappa$ of $K$, and a principal ideal $(\alpha) = \mathfrak{a}$ of $K$ coprime to $(\kappa \alpha)$, the spin of $\mathfrak{a}$ can be defined as the quadratic residue symbol
\[\left(\frac{\alpha}{\kappa \mathfrak{a}}\right) \in \{\pm 1\}.\]
For fixed  $\kappa$, these spin symbols seem to be equidistributed as $\mathfrak{a}$ varies through the principal primes of $K$. In most cases, such equidistribution is conditional on other hard results.

One exception  occurs when $\kappa$ is an involution of $K$. In this case, the symbol $\left(\frac{\alpha}{\kappa \mathfrak{a}}\right)$ is found to be determined by the splitting of the prime factors of $\mathfrak{a}$ in a certain extension of $K$. The distribution of these spin symbols is then determined by the Chebotarev density theorem.

We can explain this phenomenon with our work on theta groups, giving an alternative to the trace-based approach appearing in the proof of \cite[Proposition 12.1]{FIMR13}. Our version of this proposition is the following.

\begin{prop}
\label{prop:inv_spin}
Take $K/F$ to be a quadratic extension of number fields, and take $\kappa$ to be a generator of $\Gal(K/F)$.

Choose a nonzero integer $\alpha$ in $K$ so that the ideal $\mathfrak{a}=(\alpha)$ is coprime to $(\kappa \alpha)$. Further, assume that the extension $K(\sqrt{\alpha})/K$ splits at all primes of $K$ where $K/F$ is ramified, and  is unramified at all primes over two where $K/F$ is inert.  

Then we have the identity
\[\left(\frac{\alpha}{\kappa \mathfrak{a}}\right) =1.\]
\end{prop}
 
\begin{proof}
We will use the exact sequence
\[E \,=\, \left[0 \to \FFF_2 \xrightarrow{\,\,i \,\,} \Ind_{K/F}\FFF_2 \xrightarrow{\,\,\nu\,\,} \FFF_2 \to 0\right]\]
of $G_F$-modules, where $i$ and $\nu$ are defined as in Figure \ref{fig:nat_mor}. We take $x_0, y_0$ to be a basis for $\Ind_{K/F}\FFF_2$ chosen so $\kappa x_0 = y_0$ and $\kappa y_0 = x_0$.

Take $S$ to be the set of places of $F$ dividing some place where  $K(\sqrt{\alpha})/K$ is ramified. Note that the coprimality condition above implies that $S$ only contains places where $K/F$ splits. We define a local conditions subgroup $\msW = \tresprod{v} \msW_v$ for $\Ind_{K/F} \FFF_2$, with each $\msW_v$ defined as follows:
\begin{itemize}
\item If $K/F$ ramifies at $v$, $\msW_v$ is taken to be zero.
\item If $v$ is in $S$, $\msW_v$ is taken to be the image of $H^1(G_v, \langle x_0 \rangle)$.
\item Otherwise, $\msW_v$ is taken to be $H^1_{\text{ur}}\left(G_v, \,\Ind_{K/F} \FFF_2\right)$.
\end{itemize}

Consider the finite group
\[\mcH_{-} = \big \langle x, y, z \,\big|\,\, [x, z] = [y, z] =  x^2= y^2 =  z^2 = 1, \,\,[x,y]= z\big\rangle,\]
which is isomorphic to the dihedral group of order $8$. Define an action of $\Gal(K/F)$ on this group by
\[\kappa x = y,\quad \kappa y = x, \quad\kappa z = z.\]
Then $\mcH_{-}$   fits in an exact sequence of $G_F$-groups
\[0 \xrightarrow{\quad} \mu_2 \xrightarrow{\,\,-1\, \mapsto \,z\,\,} \mcH_{-} \xrightarrow{\,\,\substack{x \,\mapsto\, x_0 \\ y\, \mapsto\, y_0}\,\,} \Ind_{K/F}\FFF_2 \xrightarrow{\quad} 0.\]
Taking the pushout of this sequence along $\mu_2\hookrightarrow (F^s)^{\times}$ defines a theta group $\mcH$ for $\Ind_{K/F} \FFF_2$ in the sense of  \cite[Definition 5.1]{MS21}.

We now need the following lemma. As in Section \ref{sec:theta_complement_class}, we refer to \cite[Definition 5.3]{MS21} for the definition  of isotropic local conditions, and to \cite[Definition 5.7]{MS21} for the definition of the Poonen--Stoll class.
 
\begin{lem} \label{lem:theta_group_input_spin}
Given the above setup, choose $\delta \in F^{\times}$ so $K = F(\sqrt{\delta})$. Then:
\begin{enumerate}
\item The local conditions $\msW$ are isotropic with respect to the theta group $\mcH$,
\item The Poonen--Stoll class $\psi_{\textup{PS}}$ associated to the exact sequence $E$ and the theta group $\mcH$ is the quadratic character $\chi_{F,\, -\delta} \in H^1(G_F, \mu_2)$ corresponding to the extension $F(\sqrt{-\delta})/F$.
\end{enumerate}
\end{lem}
\begin{proof}
We start by establishing isotropy. For a place $v$, let $q_{\mcH,G_v}$ be as defined in \cite[Notation 5.2]{MS21}. At $v$ in $S$, $\msW_v$ is the image of 
\[H^1\big(G_v, \,\langle x\rangle\big) \subseteq H^1\big(G_v, \,\mcH_{-}\big)\]
in $H^1(G_v, \Ind_{K/F} \FFF_2)$,  so  $q_{\mcH, G_v}(\msW_v) = 0$.  At other places $v$ where $K/F$ is not ramified, $q_{\mcH, G_v}(\msW_v)$ lies in the image of $H^2(G_v/I_v, \mu_2)$, which is a trivial group.  At places $v$ where $K/F$ is ramified, $\msW_v$ is zero, so $q_{\mcH, G_v}(\msW_v) = 0$.  This proves part (1).

For the second part, we note that the map $\iota\colon \FFF_2 \to \Ind_{K/F} \FFF_2$ does not lift to a homomorphism from $\FFF_2$ to $\mcH_{-}$. However, we can define a lift of $\iota$ from $\FFF_2$ to $\mcH$ by the formula
\[1 \mapsto \sqrt{-1} \cdot x\cdot y.\]
We may use this lift to compute the Poonen--Stoll class as in \cite[Remark 5.8]{MS21}, giving the part.
\end{proof}

We return to the proof of Proposition \ref{prop:inv_spin}. With the lemma proved, \cite[Theorem 5.10]{MS21} gives
\begin{equation}
\label{eq:spin_theta}
\CTP_E(\phi_0, \,\phi_0) \,=\, \CTP_E(\phi_0, \chi_{F, \, -\delta}),
\end{equation}
for all $\phi_0$ in  $\Sel(\FFF_2, \,\pi(\msW))$, where we have identified the groups $\FFF_2$ and $\mu_2$ via the unique isomorphism between them.

For $\beta$ in $K^{\times}$, we will use the notation $\chi_{\beta}$ for the quadratic character associated to $K(\sqrt{\beta})/K$. Given $\alpha$ as in the proposition statement, take $\phi \in H^1(G_F, \Ind_{K/F}\FFF_2)$ to be the image of $\chi_{\alpha}$ under the Shapiro isomorphism.  From the conditions on $\alpha$ we see that $\phi$ maps into $\msW_v$ for all places $v$ outside $S$. Because of this, $\nu(\phi)$ lies in $\Sel(\FFF_2, \pi(\msW))$.

For $v \in S$, the restriction of $\phi$ to $G_v$ is given by the restriction of $x_0 \cdot \chi_{\alpha} + y_0 \cdot \chi_{\kappa \alpha}$ to $G_v$. Modulo $\msW_v$, this equals $\res_{G_v}\left(i(1) \cdot \chi_{\kappa \alpha}\right)$. We calculate
\begin{align*}
\CTP_E\left(\nu(\phi),\, \nu(\phi) + \chi_{F, \,-\delta}\right) \,&=\, \sum_{v \in S} \inv_v \circ \res_{G_v}\left(\chi_{\kappa \alpha}\cup \chi_{-\alpha  \cdot \kappa \alpha  \cdot \delta} \right)\\
&= \,\sum_{v \in S} \inv_v \circ \res_{G_v}\left(\chi_{\kappa \alpha}\cup \chi_{\alpha \cdot \delta} \right) \\
&=  \,\sum_{v \in S} \inv_v \circ \res_{G_v}\left(\chi_{\kappa \alpha}\cup \chi_{\alpha} \right)\quad\text{since }\sqrt{\delta} \in K\\
&=\, \frac{1}{2\pi \sqrt{-1}} \cdot \log \left(\frac{\alpha}{\kappa \mathfrak{a}}\right).
\end{align*}
The result follows from \eqref{eq:spin_theta}.
\end{proof}

\subsection{Kramer's everywhere-local/global norm group} 
\label{ssec:Kramer} 
Combining our results on field change with the results on Baer sums detailed in  \cite[Section 4.1]{MS21}, we   can  reprove a result of Kramer \cite[Theorem 2]{MR597871}. We go through this now.

We assume that  $F$ has characteristic other than $2$, take a quadratic extension $K/F$, and take $\chi$ to be the quadratic character  $\textup{Gal}(K/F)\isoarrow \{\pm 1\}$ associated to this extension.

\begin{notat}
Choose an abelian variety $A/F$. We  write $A^{\chi}$ for the quadratic twist of $A$ by $\chi$. By definition this is another abelian variety over $F$ equipped with an $K$-isomorphism
\[\psi\colon A_K\isoarrow A_K^\chi\]
 such that for any $\sigma$ in $G_F$, the composition $\psi^{-1}{}\psi^\sigma$ is multiplication by $\chi(\sigma)$ on $A$.

 Note that $\psi$ induces an isomorphism
\begin{equation*} \label{eq:2_tors_iso_twist}
\psi \colon A[2]\isoarrow A^\chi[2]
\end{equation*}
of group schemes over $F$. Via this isomorphism, we can view both  $\Sel^{2}\, A$ and $\Sel^{2}\, A^\chi$ as subgroups of $H^1(G_{F},A[2])$. 

Given a positive integer $n$, we view $A[n]$ as an element of $\SMod_F$ by equipping it with the local conditions used to define the $n$-Selmer group of $A$ over $F$, and view $A^\chi[n]$ as an element of  $\SMod_F$ similarly. We caution that although $A[2]$ and $A^\chi[2]$ are isomorphic as $G_F$-modules, they need not be isomorphic as objects in $\SMod_F$. 
With this set, the exact sequences
\[E = \Big[0 \to A[2] \xrightarrow{\,\,} A[4] \xrightarrow{\cdot 2} A[2] \to 0\Big]\]
and
\[E^{\chi} \,=\,\Big[0 \to  A^{\chi}[2] \xrightarrow{\,\,}  A^\chi[4]  \xrightarrow{\cdot 2}  A^{\chi}[2]  \to 0\Big]\]
in $\SMod_F$  give rise to the specialization of the classical Cassels--Tate pairing to the $2$-Selmer groups of $A$ and $A^{\chi}$, respectively.

\end{notat}
We are then interested in the pairing
\begin{equation}
\label{eq:kramers_pairing}
\CTP_E+\CTP_{E^\chi}\,\colon \,\Sel^{2}\, A\cap \Sel^{2}\, A^\chi\,\,\times\,\, \Sel^{2}\,A^{\vee}\cap \Sel^{2}\, (A^\chi)^{\vee} ~\to~ \mathbb{Q}/\mathbb{Z}.
\end{equation}
We will make use of the exact sequence
\[0 \to A[2] \xrightarrow{\,\,(\text{Id},\, \psi)\,\,} A \times A^{\chi} \xrightarrow{\quad} R_{K/F} A_K \to 0 \]
of group schemes over $F$, where $R_{K/F}A_K$ is the restriction of scalars of $A_K$ to $F$.  As in Example \ref{ex:res_of_scalars}, the $G_F$-module of $F^s$-points of $R_{K/F}A_K$ is $\Ind_{K/F} A(F^s)$. The final map in the above sequence is given on $F^s$-points by
\[(a, b) \,\xmapsto{\quad}\, [1] \otimes\left(a + \psi^{-1}(b)\right)\,+\, [\sigma] \otimes\sigma^{-1}\left (a- \psi^{-1}(b)\right),\]
where $\sigma$ is any element of $G_F$ outside $G_K$.

Take $M$ to be the   subgroup of $R_{K/F}A_K$ in the image of $(A \times A^{\chi})[4]$, and take 
\[\varphi\colon R_{K/F}A_K \to A \times A^{\chi}\]
 to be the associated isogeny. Equipping  $R_{K/F}A_K[\varphi]$ with the usual local conditions defining the $\varphi$-Selmer group, the sequence
\begin{equation}
\label{eq:po_kramer}
0 \to \left(A[2], \,\msW_2 + \msW_2^{\chi}\right) \xrightarrow{\,\,i\,\,}  R_{K/F}A_K[\varphi] \to A[2]\oplus A^\chi[2]  \to 0
\end{equation}
is then the pushout of $E \oplus E^{\chi}$ along the map 
\[A[2] \oplus A^{\chi}[2] \xrightarrow{\,\,(1, \psi^{-1}) \,\,} A[2].\]
Here, $i$ is defined as in Figure \ref{fig:nat_mor}, and $\msW_2$ and $\msW_2^\chi$  are the standard local conditions for $A[2]$ and $A^\chi[2]$ respectively. We then obtain by pullback an exact sequence
\begin{equation}
\label{eq:pb_kramer}
0 \to \left(A[2], \,\msW_2 + \msW_2^{\chi}\right) \xrightarrow{\,\,i\,\,}  R_{K/F}A_K[2] \xrightarrow{\,\,\nu\,\,} \left(A[2], \,\msW_2 \cap \msW^{\chi}_2\right)\to 0
\end{equation}
in $\SMod_F$,   where $R_{K/F}A_K[2]$ carries its usual local conditions, and where $\nu$ is defined as in Figure \ref{fig:nat_mor}.

The Cassels--Tate pairing of \eqref{eq:pb_kramer} equals the pairing \eqref{eq:kramers_pairing} by naturality. Applying the functor $\Sel_F$ to the sequence \eqref{eq:pb_kramer} gives the exact sequence
\[\Sel_F\left(A[2], \,\msW_2 + \msW_2^{\chi}\right) \xrightarrow{\,\, \text{res}_{G_K}^{G_F}\,\,}\Sel^2 A_K \xrightarrow{\,\, \cores^{G_K}_{G_F}\,\,} \Sel_F\left(A[2], \,\msW_2 \cap \msW_2^{\chi}\right)\]
by Shapiro's lemma and \eqref{eq:cor_from_sh}.

We are thus left with the following:
\begin{thm}[\cite{MR597871}, Theorem 2] \label{thm:kramer}
The pairing \eqref{eq:kramers_pairing} has left and right kernels
\[\textup{cor}^{G_{K}}_{G_F}~\Sel^{2}\, A_{K}\quad \quad\textup{ and }\quad\quad \textup{cor}^{G_{K}}_{G_F}~\Sel^{2}\, A^{\vee}_K\]
respectively.
\end{thm}

\begin{rmk}
In the case where $A$ is an elliptic curve, $A$ is canonically isomorphic to its dual. Thus Theorem \ref{thm:kramer} yields a non-degenerate pairing on the group
\begin{equation} \label{eq:alternating_kramer}
\left(\Sel^{2}\, A\cap \Sel^{2}\, A^\chi\right)/\textup{cor}^{G_{K}}_{G_F}~\Sel^{2}\, A_{K}.\end{equation}
This pairing is alternating since the Cassels--Tate pairings for $A$ and $A^{\chi}$ are both alternating. It follows that the group \eqref{eq:alternating_kramer} has even $\mathbb{F}_2$-dimension. This is used by Kramer to give a purely local formula for the parity of the $2$-infinity Selmer rank of $A$ over $K$ \cite[Theorem 1]{MR597871}. 

If instead $A$ is taken to be an arbitrary principally polarized abelian variety, the corresponding pairing on the group \eqref{eq:alternating_kramer} is only antisymmetric in general. Nevertheless, it is possible to use the work of Poonen--Stoll \cite{Poon99} to analyze the failure of this pairing to be alternating, allowing one to once again decompose the parity of the $2$-infinity Selmer rank of $A$ over $K$ as a sum of local terms. This is carried out in work of the first author \cite[Theorem 10.20]{MR3951582}, where the relative simplicity of \eqref{eq:pb_kramer} compared to either $E$ or $E^{\chi}$ plays an essential role.

 This simplification of the Cassels--Tate pairing under quadratic twist was also observed independently by the second author, and plays a key role in the work  \cite{Smith16b} concerning the distribution of $4$-Selmer groups in quadratic twist families. In hindsight, the underlying reason behind this simplification is that, due to \eqref{eq:pb_kramer},  the sum of the Cassels--Tate pairings $\CTP_E + \CTP_{E^{\chi}}$ can be simpler to control than either $\CTP_E$ or $\CTP_{E^{\chi}}$.
\end{rmk}

\bibliography{references}{}
\bibliographystyle{amsplain}

\newpage 
\appendix
\section{Group change operations} \label{app:group_change}
 
 This appendix reviews the construction and basic properties of group change operations for the cohomology of profinite groups. Much of the material is standard but we include it here for lack of a suitable reference (though see \cite{Hara20, Neuk08}).  We refer to \cite[ Notation 2.1]{MS21} for the notation concerning inhomogeneous cochains and cup products used below.

\subsection{Setup} \label{sec:setup_appendix}
Take $G$ to be a profinite group, and take $\text{Mod}_G$ to be the category of discrete $G$-modules. This is an abelian category with enough injectives \cite[(2.6.5)]{Neuk08}. Given $M$ in this category, the groups $H^n(G, M)$ may be defined as the right derived functors of the $G$-invariants functor  $H^0(G, \text{--})$.  The conversion between this definition and the inhomogeneous cochain definition proceeds as  follows.

Given a nonnegative integer $n$, we take $\mathscr{C}^n(G, M)$ to be the set of continuous maps from $G^{n+1}$ to $M$. These form the homogenous cochain resolution
\begin{equation*}
0 \to M \xrightarrow{\quad} \mathscr{C}^0(G, M) \xrightarrow{\,\,d\,\,}\mathscr{C}^1(G, M) \xrightarrow{\,\,d\,\,} \mathscr{C}^2(G, N) \to \dots
\end{equation*}
in $\text{Mod}_G$ detailed in \cite[I.2]{Neuk08}. The map $e$ from $C^n(G, M)$ to $\mathscr{C}^n(G, M)$ defined by
\[e(\phi)(\sigma_0, \dots, \sigma_n) = \sigma_0 \phi(\sigma_0^{-1}\sigma_1,\, \sigma_1^{-1}\sigma_2, \,\dots, \,\sigma_{n-1}^{-1}\sigma_n) \]
gives an isomorphism from $C^n(G, M)$ to $H^0(G,\mathscr{C}^n(G, M))$ and the resulting diagram
\[\begin{tikzcd}
0 \arrow{r} &C^0(G, M)\arrow{d}{\sim} \arrow{r}{d} &C^1(G, M)\arrow{d}{\sim} \arrow{r}{d} &\dots\\
0 \arrow{r} & H^0(G, \mathscr{C}^0(G, M)) \arrow{r}{d} & H^0(G, \mathscr{C}^1(G, M)) \arrow{r}{d} & \dots 
\end{tikzcd}\]
commutes. The modules $\mathscr{C}^n(G, M) $ are $H^0(G, \text{--})$-acyclic, so taking cohomology of this diagram yields the sought identification between  $H^n(G, M)$ and $R^nH^0(G, M)$.

\subsection{The basic operations} \label{sec:basic_operations}
The following gives a convenient framework for understanding certain cohomological operators on $\text{Mod}_G$.
\begin{notat}
\label{defn:cohom_op_recipe}
Take $G_0$, $G_1$, and $G_2$ to be closed subgroups of $G$. Suppose we have chosen exact functors
\[F_1 \colon \Mod_{G_0} \to \Mod_{G_1} \quad\text{and}\quad F_2\colon\Mod_{G_0} \to \Mod_{G_2} \]
that send injective modules to injective modules. Given a natural homomorphism
\[\eta\, \colon H^0(G_1, \text{--}) \circ F_1 \natmor H^0(G_2, \text{--}) \circ F_2,\]
we define the map 
\[\eta\colon H^n(G_1, F_1(M)) \to H^n(G_2, F_2(M))\] for each $M$ in $\Mod_{G_0}$, and each $n \ge 0$, by choosing an injective resolution $0 \to M \to I^{\bullet}$ of $M$ and taking the cohomology of the  diagram
\[\begin{tikzcd}
0 \arrow{r} &H^0(G_1, F_1(I^0)) \arrow{d}{\eta_{I^0}} \arrow{r} &H^0(G_1, F_1(I^1)) \arrow{d}{\eta_{I^1}} \arrow{r} &H^0(G_1, F_1(I^2)) \arrow{d}{\eta_{I^2}} \arrow{r}&\dots\\
0 \arrow{r} & H^0(G_2, F_2(I^0)) \arrow{r}  &  H^0(G_2, F_2(I^1)) \arrow{r} &  H^0(G_2, F_2(I^2)) \arrow{r} & \dots.
\end{tikzcd}\]
\end{notat}
\begin{defn}
\label{defn:exact_gc_functors}
There are three basic exact functors we will consider. First, given an open subgroup $H$ of $G$  we have the induction functor
\[\text{Ind}_G^H \colon \Mod_H \to \Mod_G\]
taking $M$ to $\Z[G] \otimes_{\Z[H]} M$.

For $H$  a closed subgroup of $G$, we have the forgetful functor
\[\text{Res}_H^G \colon \Mod_G \to \Mod_H.\]
Given a discrete $G$-module $M$, we will almost always write $\text{Res}_H^G(M)$ as $M$.

Finally, for $H$ a closed subgroup and  $\tau$ in $G$, we have the conjugation functor
\[\text{Conj}_{\tau}^H \colon \Mod_H \to \Mod_{\tau H \tau^{-1}}.\]
For each $M$ in $\text{Mod}_H$, $\text{Conj}_{\tau}^H(M)$ is defined to be $M$ as an abelian group; given $m$ in $M$  we write $\tau \cdot m$ for the corresponding element of $\text{Conj}_{\tau}^H(M)$. The $\tau H \tau^{-1}$-action is  given by
\[\sigma(\tau \cdot m) = \tau \cdot (\tau^{-1}\sigma\tau m).\]
We will usually use the notation $\tau M$ to refer to $\text{Conj}_{\tau}^H (M)$.
\end{defn}
\begin{prop}
The three functors of Definition \ref{defn:exact_gc_functors} send injectives to injectives.
\end{prop}
\begin{proof}
Conjugation is an equivalence of categories, so preserves  injectives. For $H$  open in $G$, $\Ind_G^H$ is an exact   left and right adjoint for the exact functor $\text{Res}_H^G$. So both functors preserve injectives. For   $\text{Res}_H^G$ when $H$ is not open, see \cite[Proposition 4.25]{Hara20}.
\end{proof}

In Figure \ref{fig:group_change}  we give four cohomological group change operations that can be defined using the setup of Notation \ref{defn:cohom_op_recipe}.  They give, respectively, the operations of restriction, conjugation, corestriction, and the isomorphism appearing in Shapiro's lemma.
We will also need some additional natural homomorphisms, which we give in Figure \ref{fig:nat_mor}. By functoriality of cohomology (or, equivalently, via the setup of Notation \ref{defn:cohom_op_recipe}) these morphisms induce operations on cohomology. 

\begin{notat}
Throughout Figures $1$ and $2$, $H$ denotes a subgroup of $G$, and $\tau$ denotes an element of $G$. Natural isomorphisms are marked with the notation $\natiso$. When using these objects and functors  we often suppress either $G$ or $H$ or both from the notation. By an abuse of notation we will use the lone symbol $\tau$ to refer to the functor $\Conj_{\tau}^H$, the natural isomorphism $\text{conj}_{\tau}^H$, and the natural isomorphism
\[c_{\tau}^{G/H} \circ \text{conj}_{\tau}^H \,\, \colon\,\, H^0\left(H, \text{Res}_H^G(\text{--})\right) \natiso  H^0\left(\tau H \tau^{-1}, \text{Res}_{\tau H \tau^{-1}}^G(\text{--})\right) .\]
\end{notat}

\begin{center}
\begin{figure}

 \begin{tabular}{l | l| l| l} 
      Notation & Condition & Source functor $\natmor$ Target functor & Definition   \\
      \hline 
      $\res_H^G$ & $H$ closed& $H^0(G, \text{--})  \natmor H^0\left(H, \,\text{Res}_H^G(\text{--})\right)$ & $m \mapsto m$\\
      $\text{conj}_{\tau}^H$ & $H$ closed  & $H^0(H, \text{--}) \natiso H^0\left(\tau H \tau^{-1}, \,\text{Conj}^H_{\tau}(\text{--})\right)$ & $m \mapsto \tau \cdot m$\\
      $\cores_G^H$ & $H$ open & $H^0\left(H, \,\text{Res}_H^G(\text{--})\right) \natmor H^0(G, \text{--})$ & $m \mapsto \sum_{\sigma \in G/H} \sigma m$\\
      $\shap_G^H$ & $H$ open& $H^0(H, \text{--}) \natiso H^0\left(G, \,\text{Ind}_G^H (\text{--})\right)$ & $m \mapsto \sum_{\sigma \in G/H} [\sigma] \otimes m$
      \end{tabular}

\caption{Four group change homomorphisms}
\label{fig:group_change}
\end{figure}
\begin{figure}

\begin{tabular} {l | l| l| l}
      Notation & Condition & Source functor $\natmor$ Target functor & Definition \\
      \hline
      $c^{G/H}_{\tau}$ & $H$ closed & $\text{Conj}_{\tau}^H \circ \text{Res}_{H}^G  \natiso \text{Res}_{\tau H \tau^{-1}}^G $ & $\tau \cdot m \mapsto \tau m$\\
      $\rho^{G/H}_{\tau}$ & $H$ open & $\text{Ind}_G^H \natiso \text{Ind}_G^{\tau H\tau^{-1}} \circ  \text{Conj}_{\tau}^H $ & $[\sigma] \otimes m \mapsto [\sigma\tau^{-1}] \otimes \tau \cdot m$\\[3pt]
      $i_1^{G/H}$ &   $H$ open  & $\text{Id}_{\Mod_H} \natmor \text{Res}^G_H \circ \text{Ind}_G^H$  & $m \mapsto [1] \otimes m$\\
      $i^{G/H}$ &   $H$ open & $\text{Id}_{\Mod_G} \natmor \text{Ind}_G^H\circ \text{Res}^G_H$  & $m \mapsto \sum_{\sigma \in G/H} [\sigma] \otimes \sigma^{-1} m$\\
      $\nu_1^{G/H}$ &   $H$ open & $\text{Res}^G_H \circ \text{Ind}_G^H \natmor \text{Id}_{\Mod_H}$  & $[\sigma] \otimes m \mapsto \begin{cases} \sigma m &\text{if } \sigma \in H \\ 0 &\text{otherwise.} \end{cases}$\\
      $\nu^{G/H}$ &   $H$ open & $ \text{Ind}_G^H\circ \text{Res}^G_H \natmor \text{Id}_{\Mod_G}$  & $[\sigma] \otimes m \mapsto \sigma m$\\

     \end{tabular}
\caption{Some useful natural homomorphisms}
\label{fig:nat_mor}
\end{figure}
\end{center}

Identities between the natural morphisms above automatically hold  on cohomology. Thus, for an  open subgroup $H$ of $G$, and  discrete $H$-module $M$, the diagrams 
\begin{equation}
\label{eq:sh_dec}
\begin{tikzcd}[column sep = scriptsize]
H^n(H, M) \arrow[dr, "\shap" ']\arrow[r, "i_1"] &H^n(H, \Ind_G^H M) \arrow[d, "\cores"]\\
 & H^n(G,  \Ind_G^H M)  
\end{tikzcd}\, \text{ and } 
\begin{tikzcd}[column sep = scriptsize]
H^n(H, M) &H^n(H, \Ind_G^H M) \arrow[l, "\nu_1" '] \\
 & H^n(G,  \Ind_G^H M)   \arrow[ul, "(\shap)^{-1}"] \arrow[u, "\res" ']
\end{tikzcd}
\end{equation}
commute. If $M$ is a discrete $G$-module, we  have a commutative diagram
\begin{equation}
\label{eq:cor_from_sh}
\begin{tikzcd}
H^n(H, M) \arrow[dr, "\shap"] \arrow[r,  "\cores"] &H^n(G, M) \\
H^n(G,  M) \arrow[u,"\res"] \arrow[r, "i"]  & H^n(G, \Ind_G^H M) \arrow[u,"\nu"].
\end{tikzcd}
\end{equation}
If $M$ is instead a discrete $H$-module, we have a commutative square
\begin{equation}
\label{eq:rho_tau_square}
\begin{tikzcd}
H^n(H,\, M)\arrow{r}{\tau} \arrow{d}{\shap_G^H} & H^n\left(\tau H\tau^{-1}, \,\tau M\right)\arrow{d}{\shap_G^{\tau H \tau^{-1}}} \\
H^n\left(G,  \Ind_G^H M\right ) \arrow{r}{\rho_{\tau}} & H^n\Big(G, \,\Ind_G^{\tau H\tau^{-1}} \tau M\Big)
\end{tikzcd}
\end{equation}
for every $\tau$ in $G$.

We now define a final group change operation that we will refer to as the restricted Shapiro isomorphism. This is related to the double coset formula relating restriction and corestriction (see e.g. \cite[(1.5.6)]{Neuk08}).
\begin{defn}
\label{defn:shap_GHC}
Given a profinite group $G$, take $H$ to be an open subgroup of $G$, and take $C$ to be a closed subgroup of $G$. For any $\tau$ in $G$, we take the notation
\[D_\tau = H \cap \tau^{-1} C \tau.\]
We can then define a natural morphism on the category of discrete $H$-modules
\[
r^{C\backslash G / H}_{\tau} \,\colon\, \text{Conj}^{\tau^{-1} C \tau}_{\tau} \,\circ\,\Ind^{D_{\tau} }_{\tau^{-1} C \tau}  \,\circ\,   \text{Res}_{D_{\tau}}^H  \natmor \text{Res}_C^G\,\circ\, \Ind_G^H
\]
by the formula
\[\tau\cdot\left([\sigma] \otimes m\right) \mapsto [\tau\sigma ] \otimes m.\]
If we take $B$ to be a set of representatives for the double coset $C\backslash G/H$, the map
\begin{equation}
 \label{eq:raw_doublecoset}
 r^{C\backslash G / H} \,\colon\, \bigoplus_{\tau \in B} \,\tau\left( \Ind^{ D_{\tau}}_{\tau^{-1}C\tau} ( M)\right) \isoarrow  \Ind_G^H(M)
\end{equation}
with $\tau$-component $r^{C\backslash G/H}_{\tau}$ is a natural isomorphism of $C$-modules.

We   define the \emph{restricted Shapiro isomorphism}
\begin{equation}
\label{eq:sh_GHC_def}
\restr{\shap_G^H}{C}\,\colon\,\bigoplus_{\tau \,\in\, B} H^n(D_{\tau}, \, M) \isoarrow H^n\left(C, \,\Ind_H^G M\right)
\end{equation}
by composing \eqref{eq:raw_doublecoset} with the isomorphisms
\[\text{conj}_{\tau} \circ \shap^{D_{\tau}}_{\tau^{-1} C \tau} \,\colon\,  H^n(D_{\tau}, M) \isoarrow H^n\left(C,\, \tau\left( \Ind^{ D_{\tau}}_{\tau^{-1}C\tau} ( M)\right)\right)\]
  for each $\tau$ in $B$.

When $n = 0$, this isomorphism has $\tau$ component given by
\[m \xmapsto{\quad} \sum_{\sigma \in (\tau^{-1} C \tau) / D_{\tau}} [\tau \sigma] \otimes m\, =\sum_{\sigma \in C/(\tau D_{\tau} \tau^{-1})} [\sigma \tau] \otimes m .\]
From this description, we find that the $\tau$-component of \eqref{eq:sh_GHC_def} is given by the composition
\begin{equation}
    \label{eq:shHGU_def}
H^n(D_{\tau}, M) \xrightarrow{\,i_1\,} H^n\left(D_{\tau}, \Ind_G^H M\right) \xrightarrow{\,\tau\,} H^n\left(\tau D_{\tau} \tau^{-1}, \Ind_G^H M\right) \xrightarrow{\,\cores\,} H^n\left(C, \Ind_G^H M\right),
\end{equation}
and the $\tau$ component of its inverse is the composition
\[ H^n\left(C, \Ind_G^H M\right) \xrightarrow{\,\res\,} H^n\left(\tau D_{\tau} \tau^{-1}, \Ind_G^H M\right) \xrightarrow{\,\tau^{-1}\,} H^n\left(D_{\tau}, \Ind_G^H M\right) \xrightarrow{\,\nu_1\,}
H^n(D_{\tau}, M).\]

Now suppose that $C_1$ is a closed subgroup of $C$,  and for $\epsilon$ in $G$ write
\[D_{1\epsilon} = H \cap \epsilon^{-1} C_1 \epsilon.\] 
Take $B_1$ to be a set of representatives for the double cosets in $C_1\backslash G/H$, chosen so each $\epsilon$ is of the form $c \tau$ for some $c$ in $C$ and $\tau$ in $B$. Given $\tau$ in $B$, we consider the map
\[\oplus_{\epsilon}\, \res_{D_{1\epsilon}}^{D_{\tau}} \,\colon\, H^n(D_{\tau}, \, M) \to \bigoplus_{\epsilon} H^n(D_{1\epsilon}, \, M),\]
where $\epsilon$ varies through the subset of $B_1$ in $C \tau H$. Bundling these together, we are left with a map
\[\oplus_{\tau, \epsilon}\, \res_{D_{1\epsilon}}^{D_{\tau}} \,\colon\, \bigoplus_{\tau \in B} H^n(D_{\tau}, \, M) \to \bigoplus_{\epsilon \in B_1} H^n(D_{1\epsilon}, \, M) \]
fitting into   the commutative diagram
\begin{equation}
    \label{eq:sh_HGC_C_change}
\begin{tikzcd}[row sep = large]
\bigoplus_{\tau} H^n(D_{\tau}, \, M) \arrow{r}{\restr{\shap^H_G}{C}} \arrow[d, swap, "\oplus_{\tau, \epsilon}\, \res_{D_{1\epsilon}}^{D_{\tau}}"] &   H^n\left(C, \,\Ind_H^G M\right) \arrow{d}{\res_{C_1}^C}\\
\bigoplus_{\epsilon} H^n(D_{1\epsilon}, \, M)\arrow{r}{\restr{\shap^H_G}{C_1}} & H^n\left(C_1, \,\Ind_H^G M\right).
\end{tikzcd}
\end{equation}
This is particularly useful in the case that $C = G$, where $\restr{\shap^H_G}{C}$ equals $\shap^H_G$.

If $C_1$ is a normal subgroup of $C$, and if we choose $\epsilon$ in $B_1$ corresponding to $\tau$ in $B$, we see that $D_{1\tau} = D_{1 \epsilon}$. In this case, the $\epsilon$ component of $\oplus_{\tau, \epsilon}\, \res_{D_{1\epsilon}}^{D_{\tau}}$ equals $ \res_{D_{1\tau}}^{D_{\tau}}$. Removing the redundant components of \eqref{eq:sh_HGC_C_change} then leaves a square
\begin{equation}
    \label{eq:sh_HGC_inertia}
\begin{tikzcd}
\bigoplus_{\tau} H^n(D_{\tau}, \, M) \arrow{r}{\sim} \arrow[d, swap, "\oplus_{\tau}\, \res_{D_{1\tau}}^{D_{\tau}}"] &   H^n\left(C, \,\Ind_H^G M\right) \arrow{d}{\res_{C_1}^C}\\
\bigoplus_{\tau} H^n(D_{1\tau}, \, M)\arrow[r, hook] & H^n\left(C_1, \,\Ind_H^G M\right).
\end{tikzcd}
\end{equation}
\end{defn}

\subsection{Explicit cochain formulae} \label{sec:explicit_cochain_formulae}
We detail here certain maps on inhomogeneous cochains which induce previously considered operations on cohomology.  

\subsubsection{Conjugation}

\begin{defn}
\label{defn:conj_exp}
Take $H$ to be a closed subgroup of $G$, and take $M$ to be a discrete $H$-module. Given $\tau$ in $G$ and $n \ge 0$, we have the conjugation operation
\[\text{conj}_\tau^{H} \colon C^n(H, M) \to C^n(\tau H\tau^{-1},\, \tau M)\]
defined by
\[\text{conj}_\tau^{H}(\phi)(\sigma_1, \dots, \sigma_n) = \tau \cdot \phi(\tau^{-1} \sigma_1\tau, \,\dots, \,\tau_n^{-1}\sigma_n \tau_n).\]
This gives a chain map on the complex of inhomogeneous cochains, and taking cohomology recovers the conjugation operation defined in Figure \ref{fig:group_change}.
\end{defn}

\subsubsection{Corestriction and Shapiro maps}
\begin{defn}
\label{defn:explicit_shap}
Take $H$ to be an open subgroup of  $G$, and choose a map (giving a transversal for $H$ in $G$)
\[T\colon G/H \to G\]
whose composition with the projection $G \to G/H$ is the identity, and so that $T(H)=1$.  

For $n \ge 0$, and any discrete $G$-module $M$, we define a map
\[{}^T\cores^H_G \colon C^n(H, M) \to C^n(G, M)\]
by setting
\begin{equation}
\label{eq:shap_coch}
\left({}^T\cores^H_G f\right)(\sigma_1, \dots, \sigma_n) =  \sum_{\tau H \in G/H} T(\tau H)\big( f(\sigma_1', \dots, \sigma_n')\big) 
\end{equation}
where we have taken
\[\sigma_j' = T\left(\sigma_{j-1}^{-1}\dots \sigma_2^{-1}\sigma_1^{-1}\tau H\right)^{-1}\,\, \sigma_j \,\, T\left(\sigma_{j}^{-1}\dots \sigma_2^{-1}\sigma_1^{-1}\tau H\right).\]
We have a commutative square
\begin{equation}
    \label{eq:in_to_ho_cor}
\begin{tikzcd}
C^n(H, M) \ar[r, "{}^T\cores^H_G"] \ar[d, hook] & C^n(G, M) \ar[d, hook]\\
\mathscr{C}^n(H, M) \ar{r} &\mathscr{C}^n(G, M),
\end{tikzcd}
\end{equation}
where the bottom row is the corestriction for homogeneous cochains given in \cite[I.5.4]{Neuk08},  defined from the given transversal $T$. This definition is natural, in the sense that a morphism $\alpha \colon M \to N$ of discrete $G$-modules gives a commutative square
\begin{equation}
\label{eq:shap_coch_func}
\begin{tikzcd}
C^n(H, M) \arrow[r, "{}^T\cores^H_G"] \arrow[d, "\alpha"] & C^n(G,  M)\arrow[d, "  \alpha"]\\
C^n(H, M_2) \arrow[r, "{}^T\cores^H_G"]  & C^n(G, M_2).
\end{tikzcd}
\end{equation}
We then define
\[{}^T\shap^H_G = {}^T\cores^H_G \circ i_1,\]
where $i_1$ is the natural homomorphism of Figure \ref{fig:nat_mor}. Per \eqref{eq:sh_dec}, this yields the Shapiro isomorphism when evaluated on cocycle classes. 
\end{defn}

The key properties of this construction are as follows.  
\begin{lem}
\label{lem:gc_appendix}
$\,$
\begin{enumerate}
\item For any $f$ in $C^n(H, M)$  we have
\[\label{eq:shap_coch_d}
d\left({}^T\shap^H_G f \right) = {}^T\shap^H_G (df).\]
\item  Given discrete $H$-modules $M_1, M_2$, we can define a $G$-equivariant homomorphism
\[P \colon \Ind^H_G M_1 \otimes \Ind^H_G M_2 \to \Ind^H_G (M_1 \otimes M_2)\]
by setting, for $m_1$ in $M_1$, $m_2$ in $M_2$, and $\sigma_1, \sigma_2$ in $G$,
\[P\big([\sigma_1] \otimes m_1,\,[\sigma_2] \otimes m_2\big) = \begin{cases} [\sigma_1] \otimes ( m_1 \otimes \sigma_1^{-1} \sigma_2m_2)&\text{ if } \sigma_1 H = \sigma_2 H \\ 0 &\text{ otherwise.}\end{cases}\]

Then, given $k, j \ge 0$, we have
\[\label{eq:shap_coch_cup}
{}^T\shap^H_G (f \cup g) \,=\, {}^T\shap^H_G f \,\cup_{P}\, {}^T\shap^H_G  g \]
for all $f$ in $C^k(H, M_1)$ and $g$ in $C^j(H, M_2)$.
\item Choose $\tau$ in $G$ satisfying $\tau = T(\tau H)$. Then we  have the identity
\[\nu_1 \circ \tau^{-1} \circ \res^G_{\tau H \tau^{-1}} \circ {}^T\shap^H_G (f) \,=\, f\]
for all  $f$ in $C^n(H, M)$.
\end{enumerate}

\end{lem}
\begin{proof}
Given the commutative square \eqref{eq:in_to_ho_cor}, the first part is in \cite[I.5.4]{Neuk08} for the explicit corestriction of homogeneous cochains. The identity for cup products is also immediate from the form of corestriction and cup product for homogeneous cochains, with the latter defined as in \cite[I.4]{Neuk08}.

For the third part, define functions $g_{\tau'}: G^n \to M$ so
\[{}^T\shap^H_G (f) = \sum_{\tau' \in G/H} \big[T(\tau' H)\big] \otimes g_{\tau'}.\]
For $\sigma_1, \dots, \sigma_n$ in $H$, we then have
\begin{align*}
&    \tau^{-1} \,\circ \,\res^G_{\tau H \tau^{-1}} \,\circ\, {}^T\shap^H_G (f)(\sigma_1, \dots, \sigma_n)\\
&\qquad\qquad=\, \tau^{-1} \sum_{\tau' \in G/H} \big[T(\tau' H)\big] \otimes g_{\tau'}\left(\tau \sigma_1 \tau^{-1}, \dots, \tau \sigma_n \tau^{-1}\right). 
\end{align*}
Applying $\nu_1$ just leaves
\[g_{\tau}\left(\tau \sigma_1 \tau^{-1}, \dots, \tau \sigma_n\tau^{-1}\right),\]
and the definition of the explicit Shapiro map shows this equals $f(\sigma_1, \dots, \sigma_n)$. 
\end{proof}

\end{document}